\documentclass{amsart}
\usepackage{bbm, fourier}

\def\E{{\mathbb{E}}}
\def\EQ{{\mathbf{E}}}
\def\N{{\mathbb{N}}}
\def\P{{\mathbb{P}}}
\def\PQ{{\mathbf{P}}}

\def\R{{\mathbb{R}}}

\def\Ccal{{\mathcal{C}}}

\def\Fcal{{\mathcal{F}}}
\def\Gcal{{\mathcal{G}}}

\def\Pcal{{\mathcal{P}}}

\def\Scal{{\mathcal{S}}}
\def\Tcal{{\mathcal{T}}}

\newcommand{\norm}[1]{\lVert #1 \rVert}

\newcommand{\indicator}[1]{{\mathbbm{1}}_{\{#1\}}}

\newcounter{step}

\newtheorem{theorem}{Theorem}[section]
\newtheorem{prop}[theorem]{Proposition}
\newtheorem{lemma}[theorem]{Lemma}
\newtheorem{corollary}[theorem]{Corollary}

\newcommand{\forth}[3]{\noindent \textbf{Forthcoming #1~\ref{#2}.} \emph{#3}}

\title{A Stochastic Network with Mobile Users in Heavy Traffic}

\author[S.\ Borst]{Sem Borst$^1$}
\email{sem@win.tue.nl}

\author[F.\ Simatos]{Florian Simatos$^1$}
\email{f.simatos@tue.nl}

\address{$^1$Department of Mathematics \& Computer Science \\
Eindhoven University of Technology \\ P.O. Box 513 \\
5600 MB Eindhoven, The Netherlands}

\thanks{Most of this research was carried out while the second author
was affiliated with CWI and sponsored by an NWO-VIDI grant.}

\date{\today}

\begin{document}
\maketitle

\begin{abstract}
We consider a stochastic network with mobile users in a heavy-traffic regime.
We derive the scaling limit of the multi-dimensional queue length
process and prove a form of spatial state space collapse. 
The proof exploits a recent result by Lambert and Simatos~\cite{Lambert12:0}
which provides a general principle to establish scaling limits of
regenerative processes based on the convergence of their excursions.
We also prove weak convergence of the sequences of stationary joint
queue length distributions and stationary sojourn times.
\end{abstract}

\medskip

\setcounter{tocdepth}{1}
\hrule
\vspace{-2mm}
\tableofcontents
\vspace{-8mm}
\hrule

\section{Introduction}

We consider a stochastic network with mobile users, originally
introduced in Borst et al.~\cite{Borst06:0} as a model for
a wireless data communication network.  Fluid limits of this model
were studied in Simatos and Tibi~\cite{Simatos10:0} and in this paper
we examine the heavy-traffic characteristics.

In this model, users arrive at each of the nodes according to independent
Poisson processes and then move independently of one another while still
in service.  The trajectories of the users are Markovian and governed
by an irreducible generator matrix with stationary distribution~$\pi$
(see Section~\ref{sec:notation} for notation and definitions).
At each of the nodes, users share the total capacity of the node
according to the Processor-Sharing discipline.  This assumption affects
the sojourn time distribution but not the distribution of the number of
users at the various nodes since we will restrict ourselves to the case
of exponential service requirements.  The fundamental difference between
this model and Jackson networks is that in this model, users move
independently of the service received: transitions from one node to another
are governed by the users themselves rather than by completion of service. 

\subsection*{Coexistence of two time scales and the homogenization property}

The above-mentioned feature has an important consequence: although the
rate of arrivals to and departures from the network is bounded, the rate
of movements of users within the network grows linearly with the number
of users.  The model thus shares fundamental characteristics with two
classical queueing models, namely the $M/M/1$ and the $M/M/\infty$ queues,
which are brought about in heavy-traffic conditions.
As a result, it inherits the typical time scales of both these models:
a fast time scale which governs the internal movements of users between
nodes (the $M/M/\infty$ dynamics) and a slow time scale that governs
arrivals and departures (the $M/M/1$ dynamics).
In particular, the total number of users in the network evolves slowly
compared to the speed at which users spread in the network, provided
the network is highly loaded.

One of the key technical challenges is to control the $M/M/\infty$-like
dynamics of the internal movements of users on the slow time scale of
the $M/M/1$ queue.
The main idea is that it takes a constant time for any user to be
arbitrarily close to the stationary distribution~$\pi$ while on the
other hand, in a finite time window the number of users can only have
evolved by a bounded amount because of the $M/M/1$ dynamics.
Thus starting from a large initial state, on times of order one the
total number of users will have essentially stayed the same, while each
user will be close to the stationary distribution~$\pi$.
The law of large numbers therefore suggests that starting from a large
initial state, after a constant time, users should be spread across the
various nodes according to~$\pi$, i.e., approximately a fraction~$\pi_k$
of the users should be at node~$k$.
We call this property the \emph{homogenization property}.
To derive the heavy-traffic limit, we need to show that the system
stays homogenized not only at a given fixed time, but as long as there
are a large number of users in the network.
Technically, this entails control over some hitting times, which we
achieve via martingale and coupling arguments.

\subsection*{Convergence of the full process via the convergence of excursions}

The homogenization property provides a picture of what happens when
there are a large number of users in the network: the users are spread
across the various nodes according to the stationary distribution~$\pi$,
and in particular it is unlikely for any of the nodes to be empty.
Thus, the full aggregate service rate is likely to be used,
and the total number of users evolves as in a single $M/M/1$ queue with
the combined service rate of all nodes.
This property provides a useful handle on the processes of interest far
away from zero.  Imagine for instance that the network starts empty:
it will eventually become highly loaded, at which point the homogenization
property kicks in and holds until the network becomes empty (or close to)
again.  In other words, the homogenization property should give us
control over excursions that reach a certain height and it is therefore
natural to expect the entire process to converge as well.
This line of argument has been used in Lambert et al.~\cite{Lambert11:0}
to analyze the scaling limit of the Processor-Sharing queue length process
and has later been generalized in Lambert and Simatos~\cite{Lambert12:0}.
The proofs of the scaling limit results in the present paper leverage
the general principle established in Lambert and Simatos~\cite{Lambert12:0}.

The above arguments lead to the result that the joint queue length
process asymptotically concentrates on a line whose angle corresponds
to the stationary distribution~$\pi$, thus exhibiting a form of state
space collapse.
The total number of users, after scaling, behaves asymptotically
as in a single $M/M/1$ queue, and thus evolves as a reflected
Brownian motion, with the stationary distribution converging to
an exponential distribution.
These characteristics are strongly reminiscent of the heavy-traffic
behavior of the joint queue length process in various queueing networks,
see for instance Bramson~\cite{Bramson98:0}, Reiman~\cite{Reiman84:0},
Stolyar~\cite{Stolyar04:0}, Verloop et al.~\cite{Verloop11:0}
and Williams~\cite{Williams98:0}.
However, to the best of our knowledge, this is the first result which
shows that mobility of users, rather than scheduling, routing or
load balancing, can act as a mechanism producing state space collapse.

\subsection*{Organization of the paper}

Section~\ref{sec:notation} sets up notation used throughout the paper and summarizes the main results of the paper;
it also presents two key couplings, one with an $M/M/1$ queue which
provides a lower bound and one with a closed system that is easier to
handle with regard to homogenization.  This closed system is analyzed
in Section~\ref{sec:closed}, where two bounds are derived:
one concerning the time needed for the closed system to get homogenized,
and the other concerning the time that the system stays homogenized,
starting from a homogenized state.  These estimates are used in
Section~\ref{sec:open} to derive corresponding bounds for the open system.
These bounds allow us to prove that the system is null-recurrent in
the critical case, a case that had not been treated earlier.

The last three sections then deal with the heavy-traffic regime:
Section~\ref{sec:process} proves, using the above excursion arguments,
that the sequence of processes converges weakly towards
a multi-dimensional reflected Brownian motion;
Section~\ref{sec:conv-distr} investigates the asymptotic behavior of
the stationary distributions and Section~\ref{sec:sojourn-time}
examines the asymptotic behavior of the sojourn times.

\subsection*{Relation with previous work}

Some of the results of the present paper can partially be found in Simatos
and Tibi~\cite{Simatos10:0} and we wish to explain the new contributions.
The results of Section~\ref{sec:closed} somewhat strengthen the derivations
in~\cite{Simatos10:0}: we establish tighter bounds and remove a technical
condition on the generator matrix governing the mobility of users.
One of the main results of Section~\ref{sec:open}, namely
Proposition~\ref{prop:control-homogenization}, can also be found
in~\cite{Simatos10:0} under an additional technical condition on the
generator matrix.
In the present paper we provide an alternative proof of this result
which we believe to be potentially useful in a more general setting.
All the other results are completely new.
In order to have a self-contained paper, and also because we could
significantly simplify some tedious technical details of~\cite{Simatos10:0},
we present complete (and simpler) proofs for results which were already
partially known.

\section{Notation, main results and two useful couplings}
\label{sec:notation}

Let $\N = \{0,1,\ldots\}$ be the set of nonnegative integers. In this paper we deal with multi-dimensional processes, typically taking values in $\N^K$ or $\R^K$ for some $K \geq 2$, but we also need to consider real-valued processes. It is therefore convenient to abuse notation and use the common notation $\norm{\cdot}$ to denote the $L_1$ norm on every $i$-dimensional space~$\R^i$. Thus for each $i \geq 1$ and $y \in \R^i$ we write $\norm{y} = |y_1| + \cdots + |y_i|$. In the sequel, $\Pcal(u)$ for $u > 0$ denotes a Poisson random variable with parameter~$u$. It satisfies the following large-deviation type inequality:
\begin{equation}
\label{eq:deviation-P}
\P(\Pcal(u) \geq v) \leq \exp \left( -u h(v/u) \right), \ v \geq u,
\end{equation}
where from now on $h(x) = x \log x + 1 - x$. In the sequel we will use the fact that $h(x)$ grows at least linearly as $x \to +\infty$.

\subsection{User mobility} \label{sub:mobility}

In the rest of the paper we fix some integer $K \geq 2$ and we consider
a network of $K$~nodes.  Let $\xi$ be a Markov process with state space
$\{1, \ldots, K\}$ and generator matrix
$Q = (q_{k \ell}, 1 \leq k, \ell \leq K)$, and let $\gamma > 0$ be the
trace of~$-Q$.  We assume that $Q$ is irreducible, denote by~$\pi$ its
stationary distribution, and define
$\underline \pi = \min_{1 \leq k \leq K} \pi_k > 0$
and $\overline \pi = \max_{1 \leq k \leq K} \pi_k < 1$.
We will need to measure distance to~$\pi$ and introduce the function
$\varrho: \N^K \to [0,\infty)$ defined by
\[ \varrho(y) = \left \lVert \frac{y}{\norm{y}} - \pi \right \lVert, \ y \in \N^K \]
with the convention $\varrho(y) = 0$ if $\norm{y} = 0$.
For $k = 1, \ldots, K$, let $\PQ_k$ be the law of~$\xi$ started at~$k$.
For $t \geq 0$, let $\Delta(t) =
\max_{1 \leq k, \ell \leq K} |\PQ_k(\xi(t) = \ell) - \pi_\ell|$,
so that $\Delta(t) \to 0$ as $t \to +\infty$. For $\varepsilon > 0$, we define
$\tau(\varepsilon) = \sup\{ t \geq 0: \Delta(t) \geq \varepsilon \}$.

\subsection{Sequence of networks}
\label{sub:networks}

For each $n \geq 1$,
consider $(\lambda_{n,k}, 1 \leq k \leq K) \in [0,\infty)^K$
and $(\mu_{n,k}, 1 \leq k \leq K) \in [0,\infty)^K$,
and define $\lambda_n = \lambda_{n,1} + \cdots + \lambda_{n,K}$
and $\mu_n = \mu_{n,1} + \cdots + \mu_{n,K}$.
Let $x_n = (x_{n,k}, 1 \leq k \leq K)$ be the following c\`adl\`ag, $\N^K$-valued
stochastic process: for $t \geq 0$ and $k \in \{1, \ldots, K\}$,
$x_{n,k}(t)$ is the number of users at node~$k$ at time~$t$ in the
network subject to the following dynamics:
\begin{itemize}
	\item users arrive at node~$k$ according to a Poisson process with intensity $\lambda_{n,k}$, and arrival streams are independent;
	\item users have i.i.d.\ service requirements, exponentially distributed with parameter one, independent from the arrival processes;
	\item node~$k$ serves users according to the Processor-Sharing service discipline and has capacity $\mu_{n,k}$;
	\item while still in service, users move independently from everything else according to a Markov process with generator matrix~$Q$.
\end{itemize}

According to the Processor-Sharing service discipline, the server splits its service capacity equally among all the users present at any point in time: in particular, each customer present at node $k$ at time $t$ is instantaneously served at rate $\mu_{n,k} / x_{n,k}(t)$ (provided that $x_{n,k}(t) > 0$, i.e., the node is not empty). In particular, if a customer has been present in the network between time $s$ and~$t$ and was at node $\xi(u)$ at time $u \in [s,t]$, then between $[s,t]$ it received a service equal to
\begin{equation} \label{eq:service}
	\int_s^t \frac{\mu_{n,\xi(u)}}{x_{n, \xi(u)}(u)} du.
\end{equation}

Note that the model we consider is in sharp contrast with classical queueing networks, such as Jackson networks, where customers only move upon completion of service. In the model we consider, customers arrive with a single service requirement and they are served \emph{along their route}: customers receive some service where they are and their trajectory is governed by some random dynamics independent of the service. In the model we consider, the trajectory process $\xi$ appearing in~\eqref{eq:service} is a Markov process with generator matrix $Q$, independent from all the other stochastic primitives, i.e., the arrival processes and service requirements.

Because of~\eqref{eq:service}, the stochastic process $x_n$ may in general be difficult to analyze. Nonetheless, when service requirements are exponentially distributed such as here, the memoryless property of the exponential distribution implies that $x_n$ is a Markov process with generator~$\Omega_n$ given by
\begin{multline*}
	\Omega_n(f)(y) = \sum_{k = 1}^K \lambda_{n,k} \left( f(y+e_k) - f(y) \right) + \sum_{k = 1}^K \mu_{n,k} \left( f(y-e_k) - f(y) \right) \indicator{y_k > 0}\\
	+ \sum_{1 \leq k, \ell \leq K} q_{k\ell} y_k \left( f(y-e_k+e_\ell) - f(y) \right)
\end{multline*}
for any function $f: \N^K \to \R$ and any $y = (y_k, 1 \leq k \leq K) \in \N^K$, and where $e_k$ is the $k$th unit vector of~$\N^K$. Thus $x_n$ can be seen as a system of particles, where particles are added and removed (when possible) according to independent Poisson processes attached to each node, and while alive move independently according to the same Markovian dynamics. Note also that because of the memoryless property, the Processor-Sharing assumption has no impact on the law of~$x_n$; nonetheless this assumption will impact sojourn time distributions studied in Section~\ref{sec:sojourn-time}. 

For $y \in \N^K$, let $\P_n^y$ be the law of~$x_n$ started at~$y$
(from a network perspective, users start with i.i.d.\ exponential
service requirements), and denote by $r_n$ the $K$-dimensional process
$r_n = x_n / \norm{x_n}$ with the convention $r_n(t) = \pi$ when
$\norm{x_n(t)} = 0$, and where from now on if $b \in \R$ and $y \in \R^K$
then $by \in \R^K$ denotes the vector $(by_k, 1 \leq k \leq K)$; similarly, $by$ denotes the $K$-dimensional function $(by_k, 1 \leq k \leq K)$ if $b$ is a real-valued function, with $b y_k = (b(t) y_k, t \geq 0)$.

We denote by $a_{n,k}$ the arrival process at the $k$th node and by
$d_{n,k}$ the potential departure process from the $k$th node,
so that $a_{n,k}$ is a Poisson process with intensity $\lambda_{n,k}$
and $d_{n,k}$ is a Poisson process with intensity $\mu_{n,k}$ such that
the $2K$ processes $(a_{n,k}, d_{n,k}, 1 \leq k \leq K)$ are independent.
Moreover, by definition of~$x_n$ it holds that
\begin{equation}
\label{eq:sde}
\norm{x_n(t)} = \norm{x_n(0)} + \sum_{k=1}^K a_{n,k}(t) - \sum_{k=1}^K \int_{[0,t]} \indicator{x_{n,k}(u-) > 0} d_{n,k}(du), \ t \geq 0.
\end{equation}

Let $a_n = a_{n,1} + \cdots + a_{n,K}$
and $d_n = d_{n,1} + \cdots + d_{n,K}$, so that $a_n$ and $d_n$ are
independent Poisson processes with intensities~$\lambda_n$ and~$\mu_n$,
respectively. Define $\rho_n = \lambda_n / \mu_n$: it has been proved in~\cite{Ganesh10:0,Simatos10:0} that $x_n$ is positive-recurrent when $\rho_n < 1$ and transient when $\rho_n > 1$; in Proposition~\ref{prop:critical-case} we will complete the picture and prove that $x_n$ is null-recurrent when $\rho_n = 1$. When $\rho_n < 1$ denote by~$\nu_n$ the stationary distribution of~$x_n$. We define $\kappa = \sup_{n} (\lambda_n + \mu_n)$ and assume throughout the paper that $\kappa$ is finite.

\subsection{Heavy traffic regime and main results} \label{sub:main-results} In the rest of the paper we fix two parameters $\lambda > 0$ and $\alpha \geq 0$. The main results of this paper deal with the following heavy traffic regime.
\\

\noindent \textbf{Heavy-traffic assumption.} \emph{We say that the heavy traffic assumption holds if $\rho_n \leq 1$ for each $n \geq 1$ and
\[ \lim_{n \to +\infty} \lambda_n = \lambda \ \text{ and } \ \lim_{n \to +\infty} n (1-\rho_n) = \alpha. \]
}

We believe that the techniques of the paper could be adapted to the case $\alpha \in \R$ (and hence remove the assumption $\rho_n \leq 1$). Under the heavy-traffic assumption, we have $\mu_n \to \lambda$. Note that we do not require each $\lambda_{n,k}$ or $\mu_{n,k}$ to converge, but only the corresponding sum. The heavy traffic assumption will be assumed to hold in Sections~\ref{sec:process}, \ref{sec:conv-distr} and~\ref{sec:sojourn-time}, that contain the main results of the paper. In Sections~\ref{sec:closed} and~\ref{sec:open}, we derive results on the system for fixed $n$ that do not require the heavy traffic assumption.

We now summarize the three main results of the paper: Theorems~\ref{thm:cv-process} and~\ref{thm:ht-stationary-distribution} establish the scaling limits of the sequence of processes $(x_n, n \geq 1)$ and of stationary measures $(\nu_n, n \geq 1)$, and Corollary~\ref{cor:sojourn} investigates the asymptotic sojourn time of a typical initial customer. Remember that $by = (b y_k, 1 \leq k \leq K)$ if $b$ is a real number or a real-valued function and $y \in \N^K$.
%
\\

\forth{Theorem}{thm:cv-process}{Let $X_n(t) = x_n(n^2t) / n$. If the heavy traffic assumption holds, then the sequence of processes $(X_n, n \geq 1)$ under $\P_n^{0}$ converges weakly as $n$ goes to infinity to the $K$-dimensional process $\underline B \pi$, where $B$ is a Brownian motion with drift $-\lambda \alpha$ and variance $2 \lambda$ started at $0$.}
\\

\forth{Theorem}{thm:ht-stationary-distribution}{Assume that the heavy traffic assumption holds with $\alpha > 0$, and let $X_n(0) = x_n(0)/n$. Then the sequence $(X_n(0), n \geq 1)$ under $\P_n^{\nu_n}$ converges weakly as $n$ goes to infinity to the $K$-dimensional vector $E \pi$ where $E$ is an exponential random variable with parameter~$\alpha$, and all higher moments converge as well, i.e., $\E_n^{\nu_n}(\norm{X_n(0)}^r) \to r!/\alpha^r$ for all integer $r \geq 0$.}
\\


\forth{Corollary}{cor:sojourn}{Assume that the heavy traffic assumption holds with $\alpha > 0$ and let $(y_n)$ in $\N^K$ with $\varrho(y_n) \to 0$ and $\norm{y_n} / n \to b \in (0,\infty)$. Under $\P_n^{y_n}$, let $\chi_n$ be the sojourn time of one of the $\norm{y_n}$ initial customers chosen uniformly at random. Then the sequence of random variables $(n^{-1} \chi_n)$ under $\P_n^{y_n}$ converges weakly to $b E / \lambda$, with $E$ a mean one exponential random variable.}

\subsection{Functional operators} \label{sub:func-operators}

Fix some $i \geq 1$ and $\varepsilon > 0$. Let $D_i$ be the space of $\R^i$-valued c\`adl\`ag functions. For $f = (f_k, 1 \leq k \leq i) \in D_i$, we define the following operators:
\[ T^{\downarrow}(f,\varepsilon) = \inf\{ t \geq 0: \norm{f(t)} \leq \varepsilon \}, \ T^{\uparrow}(f,\varepsilon) = \inf\{ t \geq 0: \norm{f(t)} \geq \varepsilon \} \]
as well as
\[ T_0(f) = \inf\{ t > 0: \norm{f(t)} = 0 \} \ \text{ and } \ \widetilde T_0(f) = \min_{1 \leq k \leq i} T_0(f_k). \]

For $t \geq 0$, let $\sigma, \theta_t: D_i \to D_i$ be the stopping
and shift operators, defined by $\sigma(f)(s) = f(s \wedge T_0(f))$
and $\theta_t(f)(s) = f(t+s)$, respectively, for $f \in D_i$ and $s \geq 0$.
Define also the map $e_\varepsilon^\uparrow: D_i \to D_i$ as follows:
\[ e_\varepsilon^\uparrow(f) = \left (\sigma \circ \theta_{T^\uparrow(f, \varepsilon)} \right)(f). \]
In words, $e_\varepsilon^\uparrow(f)$ is the process~$f$ shifted at the first time $T^\uparrow(f, \varepsilon)$ when $\norm{f}$ reaches level~$\varepsilon$ and stopped at the first time it reaches~0 afterwards.
Finally, let $g_\varepsilon(f)$ be the left endpoint of the first excursion of $\norm{f}$ to reach level~$\varepsilon$:
\[ g_\varepsilon(f) = \sup \left\{ t \leq T^\uparrow(f, \varepsilon): \norm{f(t)} = 0 \right\}. \]

Note that similarly as $\norm{\cdot}$, we use the same notation to refer to operators defined on functions taking values in $\R^i$ for any $i \geq 1$; for instance, we have $T^\uparrow(f, \varepsilon) = T^\uparrow(\norm{f}, \varepsilon)$.

\subsection{Coupling with an $M/M/1$ queue}

The process~$x_n$ is naturally coupled with the queue length process of
an $M/M/1$ queue with arrival rate~$\lambda_n$ and service rate~$\mu_n$.
First, for any $f \in D_1$, define the function $\underline f$,
called the function~$f$ reflected above its past infimum, by
\[ \underline f(t) = f(t) - \min \left( \inf_{0 \leq s \leq t} f(s), 0 \right). \]

For $n \geq 1$, let in the sequel $\widetilde \ell_n = \norm{x_n(0)} + a_n - d_n$ and $\ell_n = \underline {\widetilde \ell}_n$ be the process $\widetilde \ell_n$ reflected above its past infimum. Since $a_n$ and $d_n$ are independent Poisson processes with respective intensity~$\lambda_n$ and~$\mu_n$, $\widetilde \ell_n$ is a continuous-time random walk and $\ell_n$ is equal in distribution to the queue length process of an $M/M/1$ queue with arrival rate~$\lambda_n$ and departure rate~$\mu_n$.

Intuitively, this coupling does the following: the potential total
departure process from both~$x_n$ and~$\ell_n$ is given by~$d_n$.
When $d_{n,k}$ rings, there is no departure if $x_{n,k} = 0$ while
there may be other users elsewhere.  The process~$\ell_n$ ignores how
users are spread in the network: if there are users in the network
and one of the $d_{n,k}$ rings, then one of the users leaves.
Thus $\norm{x_n}$ and $\ell_n$ coincide as long as there is no empty node,
and at all times there are more departures from~$\ell_n$ than from~$x_n$.
Formally, we have the following result; recall that if $\rho_n < 1$,
then $x_n$ is positive-recurrent with stationary distribution~$\nu_n$.

\begin{lemma}
\label{lemma:coupling-MM1}
For any $y \in \N^K$, the two following properties holds $\P_n^y$-almost surely:
	\begin{itemize}
		\item $\norm{x_n(t)} = \ell_n(t)$ for all $t \leq \widetilde T_0(x_n)$;
		\item $\norm{x_n(t)} \geq \ell_n(t)$ for all $t \geq 0$.
	\end{itemize}
In particular, if $\rho_n < 1$, then $\norm{x_n(0)}$ under $\P_n^{\nu_n}$ is stochastically lower bounded by a geometric random variable with parameter~$\rho_n$.
\end{lemma}

\begin{proof}
	We prove the two first properties by induction. By construction, it holds that $\norm{x_n(0)} = \ell_n(0)$ and when $a_n$ rings, both processes $\norm{x_n}$ and $\ell_n$ increase by $1$, which preserves the difference $\norm{x_n} - \ell_n$. Internal movements in $x_n$ also preserve this difference, hence one only needs to inspect what happens when one of the $d_{n,k}$ rings.
	
	So to prove the first property, consider $t \leq \widetilde T_0(x_n)$ and $1 \leq k \leq K$ such that $d_{n,k}(\{t\}) = 1$ and assume that $\norm{x_n(t-)} = \ell_n(t-)$: we must show that $\norm{x_n(t)} = \ell_n(t)$. Since $t \leq \widetilde T_0(x_n)$, by definition of $\widetilde T_0(x_n)$ we have $x_{n,k}(t-) > 0$ and so $\norm{x_n(t)} = \norm{x_n(t)}-1$. On the other hand, we have $\ell_n(t-) = \norm{x_n(t-)} > 0$ by induction hypotheses and so by construction, $\ell_n(t) = \ell_n(t)-1$, which proves the desired property.
	
	Let us now prove the second property, so consider $t \geq 0$ and $1 \leq k \leq K$ such that $d_{n,k}(\{t\}) = 1$ and assume that $\norm{x_n(t-)} \geq \ell_n(t-)$: we must show that $\norm{x_n(t)} \geq \ell_n(t)$.  If $\ell_n(t-) = 0$ then $\ell_n(t) = 0$ and $\norm{x_n(t)} \geq \ell_n(t)$. Else, $\ell_n(t) = \ell_n(t-)-1$ and since $\norm{x_n(t)}$ decreases by at most~1 (it does when $x_{n,k}(t-)  > 0$, otherwise it stays constant) we also have $\norm{x_n(t)} \geq \ell_n(t)$ in this case.
	
	As for the second assertion of the lemma, for any $q \geq
0$ we have by stationarity and using the second property
	\[ \P_n^{\nu_n} \left( \norm{x_n(0)} \geq q \right) = \P_n^{\nu_n} \left( \norm{x_n(t)} \geq q \right) \geq \P_n^{\nu_n} \left( \ell_n(t) \geq q \right) \mathop{\longrightarrow}_{t \to +\infty} (\rho_n)^q \]
	since $\ell_n(t)$ converges in distribution as $t$ goes to infinity to a geometric random variable with parameter~$\rho_n$. This proves the result.
\end{proof}

We will often use the previous lemma in combination with the following lower bound on $\widetilde T_0(x_n)$:
\begin{equation}
\label{eq:lower-bound-tilde}
T^\uparrow(r_n - \pi, \delta) \leq \widetilde T_0(x_n), \quad \delta \leq \underline \pi.
\end{equation}

Indeed, since by definition we have $\norm{r_n(t) - \pi} \geq |r_{n,k}(t) - \pi_k|$ we see that $x_{n,k}(t) = 0$ implies that $\norm{r_n(t) - \pi} \geq \pi_k \geq \underline \pi$. In particular, this implies that $\norm{x_n(t)} = \ell_n(t)$ for all $t \leq T^\uparrow(r_n - \pi, \delta)$ and $\delta \leq \underline \pi$.

\subsection{Coupling with a closed system}
\label{sec:coupling-closed-system}

Let $x'_n$ be the process built on the same probability space as~$x_n$,
sharing the same stochastic primitives as $x_n$ but ignoring arrivals
and departures.  More precisely, if $(\xi_{n,i}, 1 \leq i \leq \norm{y})$
are the $\norm y$ independent (but not identically distributed,
due to the initial conditions) trajectories of the $\norm y$ initial
users under~$\P_n^y$, then we define
\[ x'_{n,k}(t) = \sum_{i=1}^{\norm y} \indicator{\xi_{n,i}(t) = k}, \ t \geq 0, 1 \leq k \leq K. \]
Note that $x'_n$ is a Markov process with generator~$\Omega'$ defined similarly as~$\Omega_n$ but with all $\lambda_{n,k}$'s and $\mu_{n,k}$'s equal to~0. In particular, the law of~$x_n'$ does not depend on~$n$. Under $\P_n^y$, by construction it holds that $y_k$ of the $(\xi_{n,i}, 1 \leq i \leq \norm y)$ are i.i.d.\ with common distribution~$\xi$ under~$\PQ_k$.

We define $r_n' = x'_n / \norm{x'_n}$ with the usual convention $r_n'(t) = \pi$ when $\norm{x'_n(t)} = 0$. Note that because $x'_n$ is a closed system, we have $\norm{x'_n(t)} = \norm{x'_n(0)}$. The following inequalities are intuitively clear, see for instance~\cite{Simatos10:0} for a formal proof: for any $t \geq 0$ and $k = 1, \ldots, K$,
\begin{equation}
\label{eq:coupling-inequalities}
-d_n(t) \leq x_{n,k}(t) - x'_{n,k}(t) \leq a_n(t) \ \text{ and } \ -d_n(t) \leq \norm{x_{n}(t)} - \norm{x'_{n}(t)} \leq a_n(t).
\end{equation}
This has the following useful consequence.

\begin{lemma}
\label{lemma:coupling-closed-system}
For any $y \in \N^K$ with $\norm y > 0$, we have
\[ \P_n^y \left( \forall t \geq 0: \norm{ r_{n}(t) - r'_{n}(t) } \leq \frac{2K(a_n(t) + d_n(t))}{\norm{y}} \right) = 1. \]
In particular, for any $y \in \N^K$, any $\delta > 0$ and any $t \geq 0$, we have
\begin{equation}
\label{eq:coupling-closed-system}
\P_n^y \left( \norm{ r_{n}(t) - r'_{n}(t) } \geq \delta \right) \leq \P\left( \Pcal(\kappa t) \geq \delta \norm{y} / (2K) \right).
\end{equation}
\end{lemma}

\begin{proof}
Fix $k = 1, \ldots, K$ and $y \in \N^K$ with $\norm y > 0$: then under~$\P_n^y$, it holds that
	\begin{align*}
		\left| r'_{n,k}(t) - r_{n,k}(t) \right| & = \left| \frac{x'_{n,k}(t)}{\norm y} - \frac{x_{n,k}(t)}{\norm{x_n(t)}} \right| = \left| \frac{x'_{n,k}(t) \norm{x_n(t)} - x_{n,k}(t) \norm y}{\norm y\norm{x_n(t)}} \right|\\
		& = \left| \frac{(x'_{n,k}(t) - x_{n,k}(t)) \norm{x_n(t)} + x_{n,k}(t) ( \norm{x_n(t)} - \norm y)}{\norm y\norm{x_n(t)}} \right|\\
		& \leq \frac{| x'_{n,k}(t) - x_{n,k}(t) |}{\norm y} + \frac{ \left| \norm{x_n(t)} - \norm y \right|}{\norm y}.
	\end{align*}
	
Together with~\eqref{eq:coupling-inequalities} this gives the first result, which implies~\eqref{eq:coupling-closed-system} since $a_n(t) + d_n(t)$ for any $t \geq 0$ is a Poisson random variable with parameter $(\lambda_n + \mu_n)t$ and $\lambda_n + \mu_n \leq \kappa$ by definition of~$\kappa$.
\end{proof}

\section{Analysis of the closed system}
\label{sec:closed}

In this section we are interested in the closed system~$x'_n$.
Since its law does not depend on~$n$, in order to simplify the notation
we remove temporarily all subscripts~$n$ and write $\P^y$,
$x'$, $x'_k$, $r'$ and $r'_k$ instead of $\P^y_n$, $x'_n$, $x'_{n,k}$, $r'_n$ and $r'_{n,k}$, respectively.
We will denote $x'_k(t) = \sum_{i=1}^{\norm{x'(0)}} \indicator{\xi_i(t) = k}$,
with $(\xi_i)$ independent Markov processes with generator matrix~$Q$.

\subsection{Homogenization at a fixed deterministic time}

We first show that starting from any initial state, the system becomes close to homogenization in a constant time. Note that the following bound is consistent with the central limit theorem, which suggests that $\norm{r'(t) - \pi}$ should be of order $(\norm{x'(0)})^{-1/2}$ for large $t$ and $\norm{x'(0)}$. The following result improves on Simatos and Tibi~\cite[Proposition~$5.2$]{Simatos10:0}; here we use Chernoff's instead of Chebyshev's inequality. In the following lemma, both in the statement and the proof, we make us of the functions and constants $\tau(\cdot)$, $\Delta(\cdot)$, $\overline \pi$ and $\underline \pi$ that were introduced in Section~\ref{sub:mobility}. 

\begin{lemma}
\label{lemma:bound-closed}
There exists $\varepsilon_0 > 0$, depending only on~$\pi$ and~$K$, such that for any $0 < \varepsilon < \varepsilon_0$ and any $y \in \N^K$,
\[
\P^y \left( \norm{r'(\tau(\varepsilon/(2K))) - \pi} \geq \varepsilon \right) \leq 2 K \exp \left( -\frac{\varepsilon^2 \norm{y}}{4 K^2} \right).
\]
\end{lemma}

\begin{proof}
In the rest of the proof, fix $\varepsilon > 0$, $y \in \N^K$, and write
$t = \tau(\varepsilon/(2K))$ and $\varepsilon' = \varepsilon / (2K)$. Since $r'(t) = \pi$ under~$\P^0$, the bound holds when $\norm y = 0$ and we consider $\norm y \geq 1$. Standard manipulations yield
	\[ \P^y \left( \norm{r'(t) - \pi} \geq \varepsilon \right) \leq K \max_{1 \leq k \leq K} \P^y \left(|r_k'(t) - \pi_k| \geq \varepsilon/K \right). \]
	Until the end of the proof fix some $1 \leq k \leq K$: we have $x_k'(t) = \sum_{i=1}^{\norm y} \indicator{\xi_i(t) = k}$ under~$\P^y$ and so
	\[ \left| \E^y(r_k'(t)) - \pi_k \right| = \left| \frac{1}{\norm y} \sum_{i=1}^{\norm y} \P^y(\xi_i(t) = k) - \pi_k \right| \leq \frac{1}{\norm y} \sum_{i=1}^{\norm y} \left| \P^y(\xi_i(t) = k) - \pi_k \right|. \]
	
	Thus by definition of~$\Delta$ and~$\tau$, we have $|\E^y(r_k'(t)) - \pi_k| \leq \varepsilon'$ since $t = \tau(\varepsilon')$. Consequently, the triangular inequality gives $|r_k'(t) - \pi_k| \leq |r_k'(t) - \E^y(r_k'(t))| + \varepsilon'$ and so
	\[
		\P^y \left(| r_k'(t) - \pi_k| \geq \varepsilon/K \right) \leq \P^y \left(|r_k'(t) - \E^y(r_k'(t))| \geq \varepsilon' \right) = \P^y \left( | x_k'(t) - \E^y(x_k'(t)) | \geq \varepsilon' \norm{y} \right).
	\]

Let us now define~$\varepsilon_0$.
For $p \in [0,1]$, define $f(p) = p(1-p)$.
Since $f(1-\overline \pi) > 0$, there exists
$\varepsilon_0' < 1-\overline \pi$, which only depends on
$\overline \pi$, such that $f(1-\overline \pi - \delta) > \delta$
for all $\delta < \varepsilon_0'$.
We fix such an~$\varepsilon_0'$ and consider
$\varepsilon_0 = 2K \varepsilon_0'$, which only depends
on~$\pi$ and~$K$.
In the sequel we assume that $\varepsilon < \varepsilon_0$,
or equivalently, $\varepsilon' < \varepsilon'_0$.
	
	Recall Chernoff's inequality: if $(Y_i, 1 \leq i \leq I)$ are independent random variables with $|Y_i| \leq 1$ and $\E(Y_i) = 0$ for each $1 \leq i \leq I$, then for any $0 \leq \eta \leq b$
	\[ \P \left( \left| \sum_{i=1}^I Y_i \right| \geq \eta b \right) \leq 2e^{-\eta^2/4} \ \text{ with } \ b^2 = \E(Y_1^2) + \cdots + \E(Y_I^2). \]
	
	Denote $p_i = \P^y(\xi_i(t) = k)$: we wish to apply Chernoff's inequality to the random variables $(Y_i, 1 \leq i \leq \norm{y})$ with $Y_i = \indicator{\xi_i(t) = k} - p_i$, for which $b^2 = f(p_1) + \cdots + f(p_{\norm y})$. In order to ease the notation, we suppress the dependencies of~$p_i$, $Y_i$ and~$b$ on~$k$, $t$ and~$y$, which have been fixed once and for all earlier. Define now $\eta = \varepsilon' \norm{y} / b$, so that
	\[ \P^y \left( | x_k'(t) - \E^y(x_k'(t)) | \geq \varepsilon' \norm{y} \right) = \P \left( \left| \sum_{i=1}^{\norm y} Y_i \right| \geq \eta b \right). \]
	
	Assume for a moment that $\eta \leq b$: then we could apply Chernoff's inequality and get
	\[ \P \left( \left| \sum_{i=1}^{\norm y} Y_i \right| \geq \eta b \right) \leq 2 e^{-\eta^2/4} = 2 \exp \left( -\frac{(\varepsilon' \norm{y})^2}{4 b^2} \right) \leq 2 \exp \left( -\frac{\varepsilon^2 \norm{y}}{4 K^2} \right) \]
	using $b^2 \leq \norm{y} / 4$. This would prove the result,
and so it remains only to prove that $\eta \leq b$, or equivalently,
$b^2 \geq \varepsilon' \norm{y}$. By definition we have $b^2 = f(p_1) + \cdots + f(p_{\norm y})$. Let $1 \leq i \leq \norm{y}$: since $t \geq \tau(\varepsilon')$, we have $|p_i - \pi_k| \leq \varepsilon'$ and in particular
	\[ \underline \pi - \varepsilon' \leq \pi_k - \varepsilon' \leq p_i \leq \pi_k + \varepsilon' \leq \overline \pi + \varepsilon'. \]
	
	Since $f(p) = f(1-p)$, $f$ is increasing on $[0,1/2]$ and decreasing on $[1/2,1]$ and $\underline \pi \leq 1 - \underline \pi$, the previous inequalities imply that
	\[ f(p_i) \geq \min \left( f(\underline \pi - \varepsilon'), f(\overline \pi + \varepsilon') \right) = \min \left( f(\underline \pi - \varepsilon'), f(1-\overline \pi - \varepsilon') \right) = f(1-\overline \pi - \varepsilon') \geq \varepsilon' \]
	by choice of~$\varepsilon_0'$ and since $\varepsilon' < \varepsilon_0'$. Thus $b^2 \geq \varepsilon' \norm{y}$ which concludes the proof.
\end{proof}

\subsection{Deviation time from the equilibrium}

We now study the time needed for the process~$r'$ to leave a neighborhood of~$\pi$. Note that $x'$ can be seen as a multi-dimensional Ehrenfest urn, where sharp results on hitting times in the two-dimensional case $K = 2$ have been established in Feuillet and Robert~\cite{Feuillet11:0}.

Estimates on hitting times will follow from the optional sampling theorem
applied to the martingale of Proposition~\ref{prop:martingale}.
This martingale was first constructed in Simatos and Tibi~\cite{Simatos10:0}
for the open system under an additional diagonalizability assumption on~$Q$.
Since the martingale construction is quite complicated for the open
system, we adopt here a different approach:
we only construct the martingale for the closed system and then use
coupling arguments to transfer results on the closed system to the open one.

Compared to Simatos and Tibi~\cite{Simatos10:0}, the new contribution
of the following construction consists of Lemma~\ref{lemma:key},
which makes it possible to drop the diagonalizability assumption.
Nonetheless, the construction of the martingale for the open system is
intricate while it becomes quite elementary for the closed one.
For this reason, we have chosen not to refer to~\cite{Simatos10:0},
but rather to provide a self-contained proof (except for the proof of
Lemma~\ref{lemma:finite-constant} which can be repeated almost verbatim).

\subsubsection{Additional notation}

The following notation holds throughout the rest of this section. Let $\Scal = \{u \in [0,1]^{K-1}: \norm{u} \leq 1\} \subset \R^{K-1}$ and $\Scal_K = \{ u \in [0,1]^K: \norm{u} = 1 \} \subset \R^{K}$ be the $K$-dimensional simplex. Let $L: \Scal \to \Scal_K$ be the function that completes $u \in \Scal$ into a probability distribution, i.e., $(Lu)_k = u_k$ if $1 \leq k \leq K-1$ and $(Lu)_K = 1 - \norm{u}$ for $u \in \Scal$. Note that $L$ is invertible with inverse $L^{-1}: \Scal_K \to \Scal$ being the projection of the $K-1$ first coordinates. Let finally $\Pi = \text{diag}(\pi_1, \ldots, \pi_K)$ be the diagonal matrix with entries $(\pi_k)$ on the diagonal.

Let $J$ be the Jordan normal form corresponding to~$Q$ with change of basis matrix~$\omega$. Thus $J$ and $\omega$ are possibly complex matrices, and we have the following properties, see for instance Herstein and Winter~\cite{Herstein88:0}. Let $(\vartheta_i, 1 \leq i \leq I)$ be the $I$~distinct eigenvalues of~$Q$ with $\vartheta_I = 0$, for $1 \leq i \leq I$ let $m_i$ be the algebraic multiplicity of $\vartheta_i$ and let $\omega_{k(i)}$ for some $k(i) \in \{ 1, \ldots, K \}$ be any eigenvector of~$J$ corresponding to the eigenvalue~$\vartheta_i$, i.e., $J \omega_{k(i)} = \vartheta_i \omega_{k(i)}$. Since $m_i$ is the algebraic multiplicity of~$\vartheta_i$ we have in particular $m_1 \vartheta_1 + \cdots + m_{I-1} \vartheta_{I-1} = -\gamma$ (recall that $\gamma > 0$ is the trace of~$-Q$). Moreover, we have $Q = \omega^{-1} J \omega$ so that $e^{tQ} = \omega^{-1} e^{tJ} \omega$ for any $t \in \R$, and because of the block structure of~$J$ this gives $(e^{tJ})_{k(i), j} = 0$ for $j \neq k(i)$ and $(e^{tJ})_{k(i), k(i)} = e^{t \vartheta_i}$. In the sequel we consider the following function $F: \R^K \to [0,\infty)$:
\[ F(u) = \prod_{i=1}^{I-1} \left| (\omega u)_{k(i)} \right|^{m_i}, \ u \in \R^K. \]

Then $F$ satisfies the following simple property, which is key to generalize the martingale construction of Simatos and Tibi~\cite{Simatos10:0} to the case of non-diagonalizable~$Q$.

\begin{lemma}
\label{lemma:key}
For any $u \in \R^K$ and $t \in \R$, we have $F(e^{tQ} u) = e^{-\gamma t} F(u)$.
\end{lemma}

\begin{proof}
	We have $e^{tQ} = \omega^{-1} e^{tJ} \omega$ so that
	\[ F(e^{tQ} u) = \prod_{i=1}^{I-i} \left| (\omega e^{tQ} u)_{k(i)} \right|^{m_i} = \prod_{i=1}^{I-1} \left| (e^{tJ} \omega u)_{k(i)} \right|^{m_i} = \prod_{i=1}^{I-1} \left| e^{t\vartheta_i} (\omega u)_{k(i)} \right|^{m_i} = e^{-\gamma t} F(u). \]
	The identity $(e^{tJ} \omega u)_{k(i)} = e^{t \vartheta_i} (\omega u)_{k(i)}$ follows from the fact that $(e^{tJ})_{k(i), j}$ is equal to~0 for $j \neq k(i)$ and to $e^{t \vartheta_i}$ for $j = k(i)$. This proves the result.
\end{proof}

\subsubsection{Upper bound on hitting times}

Proposition~\ref{prop:closed-system} is the main technical result of this section, which will be used in the following section. We omit the proof of the following result, for which one can repeat almost
verbatim the proof of Lemma~A.5 in Simatos and Tibi~\cite{Simatos10:0}.

\begin{lemma}
\label{lemma:finite-constant}
The quantity $\sup_{0 < c < 1} \left( c^K \int_\Scal (F(\Pi^{-1} L u))^{c-1} du \right)$ is finite.
\end{lemma}

\begin{prop}
\label{prop:martingale}
For any $c > 0$ and any $y \in \N^K$, the process
	\[ M_c(t) = e^{- c \gamma t} \int_{\Scal} \prod_{k=1}^K \left( \frac{(Lu)_k}{\pi_k} \right)^{x'_k(t)} \left(F(\Pi^{-1} L u)\right)^{c-1} du, \ t \geq 0, \]
is a bounded martingale under~$\P^y$.
\end{prop}

\begin{proof}
	Fix in the rest of the proof some $c > 0$ and $y \in \N^K$. Then under~$\P^y$, we have for any $t \geq 0$
	\[ M_c(t) \leq \underline \pi^{-\norm{y}} \int_{\Scal} \left(F(\Pi^{-1} L u)\right)^{c-1} du \]
	and so $\sup_{t \geq 0} M_c(t)$ is bounded by Lemma~\ref{lemma:finite-constant} for $0 < c < 1$ while for $c \geq 1$ one only needs to use the fact that $u \in \Scal \mapsto F(\Pi^{-1} L u)$ is bounded. Let $s,t \geq 0$: we have
	\[ \E^y(M_c(t+s) \, | \, \Fcal_t) = e^{-c \gamma (t+s)} \int_{\Scal} \prod_{k=1}^K \E^y \left[ \exp \left( \sum_{k=1}^K G_k(u) x'_k(t+s) \right) \, \Big | \, x'(t) \right] \left(F(\Pi^{-1} L u)\right)^{c-1} du \]
	with $G_k(u) = \log ((Lu)_k / \pi_k)$, so that $e^{G_k(u)} = (\Pi^{-1} L u)_k$. For any $z \in \N^K$, we have
	\begin{align*}
		\E^z \left[ \exp \left( \sum_{k=1}^K G_k(u) x'_k(s) \right) \right] & = \E^z \left[ \exp \left( \sum_{k=1}^K \sum_{i=1}^{\norm{z}} G_k(u) \indicator{\xi_i(s) = k} \right) \right]\\
		& = \prod_{i=1}^{\norm{z}} \E^z \left[ \exp \left( \sum_{k=1}^K G_k(u) \indicator{\xi_i(s) = k} \right) \right]
	\end{align*}
	since the $(\xi_i, 1 \leq i \leq \norm{z})$ under~$\P^z$ are independent. For $j \in \{1, \ldots, K\}$, $z_j$ of the $(\xi_i)$ are i.i.d.\ with distribution $\xi$ under~$\PQ_{j}$ and so
	\begin{align*}
		\prod_{i=1}^{\norm{z}} \E^z \left[ \exp \left( \sum_{k=1}^K G_k(u) \indicator{\xi_i(s) = k} \right) \right] & = \prod_{j=1}^K \left\{ \EQ_{j} \left[ \exp \left( \sum_{k=1}^K G_k(u) \indicator{\xi(s) = k} \right) \right] \right\}^{z_j}\\
		& = \prod_{j=1}^K \left( \sum_{k=1}^K e^{G_k(u)} \PQ_{j} \left( \xi(s) = k \right) \right)^{z_j}\\
		& = \prod_{j=1}^K \left( \sum_{k=1}^K (e^{sQ})_{jk} (\Pi^{-1} L u)_k \right)^{z_{j}} = \prod_{j=1}^K \left\{ \left( e^{sQ} \Pi^{-1} L u \right)_{j} \right\}^{z_j}.
	\end{align*}
	Thus
	\[ \E^y(M_c(t+s) \, | \, \Fcal_t) = e^{-c \gamma (t+s)} \int_{\Scal} \prod_{k=1}^K \left\{ \left( e^{sQ} \Pi^{-1} L u \right)_{k} \right\}^{x'_k(t)} \left(F(\Pi^{-1} L u)\right)^{c-1} du. \]
	
	We want to make the change of variables $(Lv)_k / \pi_k = (e^{sQ} \Pi^{-1} L u)_k$, i.e., $\Pi^{-1} L v = e^{sQ} \Pi^{-1} L u$ or $L v = \Pi e^{sQ} \Pi^{-1} L u$. To use the inverse of~$L$, we check that $\Pi e^{sQ} \Pi^{-1} L u \in \Scal_K$ for $u \in \Scal$: we have $(\Pi e^{sQ} \Pi^{-1} L u)_k \geq 0$ because $Lu$ has only positive coordinates and the matrix $\Pi e^{sQ} \Pi^{-1}$ has only positive coefficients, and
	\begin{multline*}
		\sum_{k=1}^K (\Pi e^{sQ} \Pi^{-1} L u)_k = \sum_{k=1}^K \pi_k \sum_{j=1}^K (e^{sQ})_{kj} \frac{(Lu)_{j}}{\pi_{j}} = \sum_{j=1}^K \frac{(Lu)_{j}}{\pi_j} \sum_{k=1}^K \pi_k \PQ_k(\xi(s) = j)\\
		= \sum_{j=1}^K \frac{(Lu)_{j}}{\pi_j} \pi_{j} = \sum_{k=1}^K (Lu)_k = 1.
	\end{multline*}

	Hence we can consider $v = H_s u$ with $H_s = L^{-1} \Pi e^{sQ} \Pi^{-1} L$. Clearly $H_s$ is invertible with inverse $H_s^{-1} = L^{-1} \Pi e^{-sQ} \Pi^{-1} L = H_{-s}$. For any $t \in \R$, we have $H_t(\Scal) \subset \Scal$: indeed, $H_t u$ for $u \in \Scal$ has only non-negative coordinates, because of the same arguments as above, and moreover
	\[
		\sum_{k = 1}^{K-1} (H_t u)_k = \sum_{k = 1}^{K-1} (L^{-1} \Pi e^{tQ} \Pi^{-1} L u)_k = \sum_{k = 1}^{K-1} (\Pi e^{tQ} \Pi^{-1} L u)_k \leq \sum_{k = 1}^{K} (\Pi e^{tQ} \Pi^{-1} L u)_k = 1.
	\]
	Hence $H_s(\Scal) \subset \Scal$ and $H_{-s}(\Scal) = H_s^{-1}(\Scal) \subset \Scal$ and so the restriction $H_s: \Scal \to \Scal$ of $H_s$ to $\Scal$  is invertible with inverse $H_{-s}$. This gives
	\[ \E^y\left(M_c(t+s) \, | \, \Fcal_t\right) = e^{-c \gamma (t+s)} \int_{\Scal} \prod_{k=1}^K \left( \frac{(L v)_k}{\pi_k} \right)^{x'_k(t)} (F(\Pi^{-1} L H_{-s} v))^{c-1} |\text{Jac}_v(H_{-s})| dv \]
	where if $M$ is a matrix $|M|$ stands for its determinant and if $M: \R^a \to \R^b$ then $\text{Jac}_v(M)$ stands for its Jacobian matrix evaluated at~$v$:
	\[ \text{Jac}_v(M) = \left( \frac{\partial M_k}{\partial u_\ell}(v), 1 \leq k \leq b, 1 \leq \ell \leq a \right). \]
	
	Lemma~\ref{lemma:key} therefore gives, since $\Pi^{-1} L H_{-s} v = e^{-sQ} \Pi^{-1} L v$,
	\[ \E^y(M_c(t+s) \, | \, \Fcal_t) = e^{-c \gamma t} e^{-\gamma s} \int_{\Scal} \prod_{k=1}^K \left( \frac{(L v)_k}{\pi_k} \right)^{x'_k(t)} \left( F(\Pi^{-1} L v)\right)^{c-1} |\text{Jac}_v(H_{-s})| dv. \]
	
	We now show that $|\text{Jac}_v(H_{-s})| = e^{\gamma s}$, which will complete the proof. The chain rule gives
	\[ \text{Jac}_v(H_{-s}) = \text{Jac}_{Lv}(L^{-1}\Pi e^{-sQ} \Pi^{-1}) \text{Jac}_v(L) = \text{Jac}_{Lv}(\Pi e^{-sQ} \Pi^{-1}) \text{Jac}_v(L) \]
	using for the second equality that $L^{-1}$ is the projection on the first $K-1$ coordinates. Also, $\text{Jac}_v (L)$ is the diagonal matrix $\text{diag}(1,1,\ldots,1,-1)$, so its determinant is equal to $-1$ and on the other hand, if $M$ is a linear operator then $\text{Jac}_v(M) = M$ for all $v$ and so
	\[ |\text{Jac}_{Lv}(\Pi e^{-sQ} \Pi^{-1})| = |\Pi e^{-sQ} \Pi^{-1}| = |e^{-sQ}| = e^{\gamma s}. \]
	The proof is complete.
\end{proof}

The following proposition makes use of the function and constants $\varrho(\cdot)$, $\gamma$ and $\underline \pi$ defined in Section~\ref{sub:mobility} and of the operator $T^\uparrow$ defined in Section~\ref{sub:func-operators}. 

\begin{prop}
\label{prop:closed-system}
There exists a family of finite constants $(\Ccal_{\delta}, \delta > 0)$ such that for every $t > 1/\gamma$, every $\delta > 0$ and every $y \in \N^K$ with $\varrho(y) \leq \underline \pi \delta^2 / 8$,
\[
\P^y \left( T^\uparrow(r'-\pi, \delta) \leq t \right) \leq \Ccal_{\delta} \exp \left( K \log t - \delta^2 \norm{y} / 8 \right).
\]
\end{prop}

\begin{proof}
	Fix in the rest of the proof $\delta > 0$, $t > 1/\gamma$
and $y \in \N^K$, and denote $c = 1/(\gamma t)$
and $T = T^\uparrow(r'-\pi, \delta)$: then Markov inequality gives
	\[ \P^y \left( T \leq t \right) \leq e \E^y \left( \exp \left( - c \gamma T \right)\right). \]
	
	We now derive an upper bound on this Laplace transform. Consider the martingale~$M_c$ of Proposition~\ref{prop:martingale}. Since under~$\P^y$ it is a bounded martingale by Proposition~\ref{prop:martingale} the optional stopping theorem gives
	\[ \E^y \left( M_c(T) \right) = \E^y \left( M_c(0) \right). \]

	We will provide an upper bound on $\E^y \left( M_c(0) \right)$ and a lower bound on $\E^y \left( M_c(T) \right)$ of the form $A \E^y(e^{-c\gamma T})$, thus providing a desired upper bound on $\E^y(e^{-c\gamma T})$. For $u, v \in \Scal_K$ let $H(u,v) = u_1 \log(u_1 / v_1) + \cdots + u_K \log(u_K / v_K)$ be the relative entropy, so that $M_c$ under~$\P^y$ can be rewritten as
	\[ M_c(t) = e^{-c\gamma t} \int_\Scal \exp \left\{ \norm{y} \left( H(r'(t), \pi) - H(r'(t), Lu) \right) \right\} \left( F(\Pi^{-1} L u) \right)^{c-1} du. \]
	
	The relative entropy is positive and satisfies the following upper and lower bounds:
\begin{equation}
\label{eq:bounds-relative-entropy}
\frac{1}{2}\norm{u - v}^2 \leq H(u, v) \leq \frac{1}{\min_k v_k} \norm{u - v}, \ u, v \in \Scal_K.
\end{equation}
	
	The lower bound, called Pinsker's inequality, is well-known, see for instance Pinsker~\cite{Pinsker64:0}, while the upper bound follows by convexity:
	\[ H(u,v) = \sum_{k=1}^K u_k \log \left( 1 + \frac{u_k-v_k}{v_k} \right) \leq \sum_{k=1}^K \frac{u_k(u_k-v_k)}{v_k} \leq \sum_{k=1}^K \frac{u_k|u_k-v_k|}{v_k} \leq \frac{\norm{u - v}}{\min_k v_k}, \]
	using $u_k \leq 1$. Thus using that the relative entropy is always positive, we obtain using the upper bound in~\eqref{eq:bounds-relative-entropy}
	\[ \E^y \left( M_c(0) \right) \leq e^{\norm{y} H(y/\norm{y}, \pi)} \int_\Scal \left( F(\Pi^{-1} L u) \right)^{c-1} du \leq A_1 e^{ \norm{y} \varrho(y) / \underline \pi} c^{-K} = A_2 e^{K \log t + \norm{y} \varrho(y) / \underline \pi} \]
	with $A_1 = \sup_{0 < c < 1}( c^K \int_\Scal (F(\Pi^{-1} L u))^{c-1} du )$ and $A_2 = A_1 \gamma^K$, $A_1$ and therefore $A_2$ being finite by Lemma~\ref{lemma:finite-constant} (recall that $0 < c < 1$ by assumption). We now derive a lower bound on $M_c(T)$. The lower bound in~\eqref{eq:bounds-relative-entropy} gives $H(r'(T), \pi) \geq \norm{x(T) - \pi}^2 / 2$ and since $\norm{x(T) - \pi} \geq \delta$ by definition of~$T$ we obtain the following lower bound on the integral part of $M_c(T)$:
	\begin{multline*}
		\int_\Scal e^{\norm{y} \left( H(r'(T), \pi) - H(r'(T), Lu) \right)} \left( F(\Pi^{-1} L u) \right)^{c-1} du\\
		\geq e^{\norm{y} \delta^2/2} \int_\Scal e^{-\norm{y} H(r'(T), Lu)} \left( F(\Pi^{-1} L u) \right)^{c-1} du.
	\end{multline*}
	
	By definition, $F$ is a continuous function, hence it is bounded on the compact set $\Scal' = \Pi^{-1} L(\Scal)$ and so
	\[
		\int_\Scal e^{-\norm{y} H(r'(T), Lu)} \left( F(\Pi^{-1} L u) \right)^{c-1} du \geq A_3 \int_\Scal e^{-\norm{y} H(r'(T), Lu)} du
	\]
	with $A_3 = (1 + \sup_{\Scal'} F)^{-1}$. For $v \in \Scal_K$ let $\Scal(v) = \{ u \in \Scal_K: H(v, u) \leq \delta^2/4 \}$ and $\phi(v) = \int_{\Scal(v)} du$ be the volume of $\Scal(v)$: then
	\[
		\int_\Scal e^{-\norm{y} H(r'(T), Lu)} du \geq e^{ -\norm{y} \delta^2/4} \phi(r'(T)) \geq e^{-\norm{y} \delta^2/4} A_4
	\]
	with $A_4 = \inf_{v \in \Scal_K} \phi(v)$. We have $A_4 > 0$ since $\phi$ is easily seen to be continuous, $\phi(v) > 0$ for any $v \in \Scal_K$ and $\Scal_K$ is compact. Therefore, we have proved
	\[ M_c(T) \geq e^{-c \gamma T} e^{\norm{y} \delta^2 / 2} A_3 e^{-\norm{y} \delta^2 / 4} A_4 = A_5 e^{-c\gamma T} e^{\norm{y} \delta^2/4} \]
	with $A_5 = A_3 A_4$ a deterministic constant.
Combining all the previous bounds, we obtain
	\[ \P^y \left( T^\uparrow(r'-\pi, \delta) \leq t \right) \leq \frac{e A_2}{A_5} \exp\left(K \log t + \norm{y} \varrho(y) / \underline \pi - \norm{y} \delta^2/4\right). \]
	Assuming $\varrho(y) \leq \underline \pi \delta^2 / 8$ ends the proof, since $A_2$ and $A_5$ only depend on~$\delta$.
\end{proof}

\section{Analysis of the open system}
\label{sec:open}

We extend Lemma~\ref{lemma:bound-closed}
and Proposition~\ref{prop:closed-system} to the open system, exploiting
the coupling with~$x'_n$ of Section~\ref{sec:coupling-closed-system}.
In the rest of the paper, $\varepsilon_0$ denotes the constant given by
Lemma~\ref{lemma:bound-closed}. 
Recall that $\kappa = \sup_{n \geq 1} (\lambda_n + \mu_n)$ is assumed
to be finite, and that the various constants, functions and operators that will be used, e.g., $\underline \pi$, $\gamma$, $\tau(\cdot)$, $\varrho(\cdot)$, $T^\uparrow$ and $T^\downarrow$ have been defined in Sections~\ref{sub:mobility} and~\ref{sub:func-operators}. 

\begin{lemma}
\label{lemma:bound-open}
For any $n \geq 1$, any $\delta < 2\varepsilon_0$ and any $y \in \N^K$, we have
\[ \P_n^y \left( \norm{r_n(\tau) - \pi} \geq \delta \right) \leq \P \left( \Pcal(\kappa \tau) \geq \delta \norm{y} / (4K) \right) + 2K \exp \left( - \frac{\delta^2 \norm{y}}{16 K^2} \right) \]
where $\tau = \tau(\delta/(4K))$
\end{lemma}

\begin{proof}
	Since
	\[ \P_n^y \left( \norm{r_n(\tau) - \pi} \geq \delta \right) \leq \P_n^y \left( \norm{r_n(\tau) - r'_n(\tau)} \geq \delta / 2 \right) + \P_n^y \left( \norm{r_n'(\tau) - \pi} \geq \delta / 2 \right) \]
	the result follows directly from Lemmas~\ref{lemma:coupling-closed-system} and~\ref{lemma:bound-closed}.
\end{proof}

\begin{prop}
\label{prop:control-homogenization}
	 There exists a finite constant $c_1 > 0$ and for each $\delta > 0$, there exists a finite constant $c_2(\delta)$ such that $c_1$ and $c_2(\delta)$ only depend on~$K$, $\kappa$ and~$Q$ (and $\delta$ for $c_2(\delta)$)  and such that for every $n \geq 1$, every $t > 0$ and every $0 < \delta < 1$, $\phi \in \N$ and $y \in \N^K$ such that:
	\[ \eta < 2\varepsilon_0, \quad \frac{\phi \eta}{8 K \kappa} > \max \left( \tau\left(\eta / (4K)\right), \frac{1}{\gamma} \right), \quad \norm{y} > \phi \quad \text{ and } \quad \varrho(y) \leq \eta, \]
	where $\eta = \underline \pi \delta^2 / 32$, then
	\[
		\P_n^y \left( T^\uparrow(r_n-\pi, \delta) \leq t \wedge T^\downarrow(x_n, \phi) \right) \leq c_2(\delta) \exp\left( \log t - c_1 (\delta^4 \phi - \log \phi) \right).
	\]
\end{prop}

\begin{proof}
In the rest of the proof, fix $n$, $t$, $\delta$, $\phi$, $\eta$ and~$y$
as in the statement of the proposition and denote for simplicity
$\tau = \tau(\eta / (4K))$ and $u = \phi \eta / (8K\kappa)$.  We have
	\[
		\P_n^y \left( T_r \leq t \wedge T_x \right) \leq \sum_{i = 0}^{\lfloor t/u \rfloor -1} \P_n^y \left( T_r \leq T_x, iu \leq T_r \leq (i+1)u \right)
	\]
	where from now on $T_r = T^\uparrow(r_n-\pi, \delta)$ and $T_x = T^\downarrow(x_n, \phi)$. Since $\norm{y} \geq \phi$ and $\varrho(y) \leq \eta$, the term corresponding to $i = 0$ in the above sum is upper bounded by
	\[ \P_n^y \left( T_r \leq T_x, T_r \leq u \right) \leq \sup_{y': \norm{y'} \geq \phi, \varrho(y') \leq \eta} \P_n^{y'} \left( T_r \leq u \right). \]
	
	Consider now $i > 0$, and note that $iu \geq \tau$ by assumption. Since $\norm{x_n(iu - \tau)} \geq \phi$ in the event $\{ iu \leq T_x \}$, the Markov property at time $iu-\tau$ gives
	\[ \P_n^y \left( T_r \leq T_x, iu \leq T_r \leq (i+1)u, \norm{r_n(iu) - \pi} \geq \eta \right) \leq \sup_{y': \norm{y'} \geq \phi} \P_n^{y'} \left( \norm{r_n(\tau) - \pi} \geq \eta \right). \]
	
	Similarly, since $\norm{x_n(iu)} \geq \phi$ in the event $\{ iu \leq T_x \}$, the Markov property at time~$iu$ gives
	\[
		\P_n^y \left( T_r \leq T_x, iu \leq T_r \leq (i+1)u, \norm{r_n(iu) - \pi} \leq \eta \right) \leq \sup_{y': \norm{y'} \geq \phi, \varrho(y') \leq \eta} \P_n^{y'} \left( T_r \leq u \right).
	\]
	
	Since the previous upper bounds do not depend on~$i$, summing over $0 \leq i \leq \lfloor t/u \rfloor-1$ gives
	\[ \P_n^y \left( T_r \leq t \wedge T_x \right) \leq \lfloor t/ u \rfloor \left( \sup_{y': \norm{y'} \geq \phi} \P_n^{y'} \left( \norm{r_n(\tau) - \pi} \geq \eta \right) + \sup_{y': \norm{y'} \geq \phi, \varrho(y') \leq \eta} \P_n^{y'} \left( T_r \leq u \right) \right).
	\]
	
	Let $y' \in \N^K$ with $\norm{y'} \geq \phi$: since $\eta < 2 \varepsilon_0$ and $\tau = \tau(\eta / (4K))$, Lemma~\ref{lemma:bound-open} gives
	\[ \P_n^{y'} \left( \norm{r_n(\tau) - \pi} \geq \eta \right) \leq \P \left( \Pcal(\kappa \tau) \geq \eta \phi / (4K) \right) + 2K \exp \left( - \frac{\eta^2 \phi}{16 K^2} \right). \]
	
	Since $\phi \geq 8 \kappa K \tau / \eta$, we have
	\[ \P \left( \Pcal(\kappa \tau) \geq \eta \phi / (4K) \right) \leq \P\left(\Pcal(\phi \eta / (8 K)) \geq \phi \eta / (4 K) \right) \leq \exp\left(-\frac{\phi \eta h(2)}{8 K}\right) \]
	and we finally get
	\[ \sup_{y': \norm{y'} \geq \phi} \P_n^{y'} \left( \norm{r_n(\tau) - \pi} \geq \eta \right) \leq \exp\left(-\frac{\phi \eta h(2)}{8 K}\right) + 2K \exp \left( -\frac{\eta^2 \phi}{16 K^2} \right). \]
	
	Consider now $y' \in \N^K$ with $\norm{y'} \geq \phi$ and $\varrho(y') \leq \eta$: then
	\[
		\P_n^{y'} \left( T_r \leq u \right) \leq \P_n^{y'} \left( T^\uparrow(r_n-r_n', \delta/2) \leq u \right) + \P_n^{y'} \left( T^\uparrow(r_n'-\pi, \delta/2) \leq u \right).
	\]
	
	On the one hand, Lemma~\ref{lemma:coupling-closed-system} implies that
	\[ \P_n^{y'} \left( T^\uparrow(r_n-r_n', \delta/2) \leq u \right) \leq \P_n^{y'} \left( T^\uparrow(a_n + d_n, \delta \norm{y'}/(4K)) \leq u \right) \leq \P \left( \Pcal(\kappa u) \geq \eta \norm{y'}/(4K) \right) \]
	using that $\delta > \eta$ and that $a_n + d_n$ is an increasing process to get the last inequality. Since $\norm{y'} \geq \phi$ this gives
	\[ \P_n^{y'} \left( T^\uparrow(r_n-r_n', \delta/2) \leq u \right) \leq \P \left( \Pcal(\kappa u) \geq \delta \phi/(4K) \right) \leq \exp\left(-\frac{\phi \eta h(2)}{8 K}\right) \]
	plugging in the definition of~$u$ and using~\eqref{eq:deviation-P} for the last inequality. On the other hand, since $u > 1/\gamma$ and $\varrho(y') \leq \eta = \underline \pi (\delta/2)^2 / 8$, Proposition~\ref{prop:closed-system} implies that
	\begin{align*}
		\P_n^{y'} \left( T^\uparrow(r_n'-\pi, \delta/2) \leq u \right) & \leq \Ccal_{\delta/2} \exp \left( K \log u - \delta^2 \norm{y'} / 32 \right)\\
		& \leq \Ccal_{\delta/2} (\eta / (8 K\kappa))^K \exp \left( K \log \phi - \delta^2 \phi / 32 \right)
	 \end{align*}
	using $\norm{y'} \geq \phi$ and the definition of $u$ to obtain the last inequality. Gathering the previous upper bounds, we have proved at this point that
	\begin{multline*}
		\frac{u}{t} \P_n^y \left( T^\uparrow(r_n-\pi, \delta) \leq t \wedge T^\downarrow(x_n, \phi) \right) \leq 2K \exp \left( -\frac{\eta^2 \phi}{16 K^2} \right) + 2\exp\left(-\frac{\phi \eta h(2)}{8 K}\right)\\
		+ \Ccal_{\delta/2} (\eta / (8 K\kappa))^K \exp \left( K \log \phi - \delta^2 \phi / 32 \right).
	\end{multline*}
	The result then follows easily from this expression.
\end{proof}

Proposition~\ref{prop:control-homogenization} has the following consequence, which is interesting in its own right and will be used in the proof of Lemma~\ref{lemma:crude-upper-bound}.

\begin{prop} \label{prop:critical-case}
	Let $n \geq 1$: if $\rho_n = 1$ then $x_n$ is null-recurrent.
\end{prop}

\begin{proof}
	Since $\ell_n$ is null-recurrent it is enough to show that $x_n$ is recurrent by Lemma~\ref{lemma:coupling-MM1}. Intuitively, the drift of~$x_n$ is always strictly positive due to the fact that there is always a positive probability that a potential departure finds an empty node, creating a slack between the arrival and departure rates. However, the drift should go to~0 as the size of the initial state becomes large. Theorem~3.2 in Lamperti~\cite{Lamperti60:0} asserts that if the drift vanishes sufficiently fast, then $x_n$ is recurrent.
	
	Since $n$ is fixed, in order to ease notation, we omit the subscript~$n$, so for instance $a$ and $d$ refer to the processes~$a_n$ and~$d_n$, respectively, and $x$ refers to~$x_n$. Let $b = 1/4$ and $\varphi(y) = \norm{y}^{1 + 2b} = \norm{y}^{3/2}$ for $y \in \N^K$. For $i \geq 1$ let $\omega_i = \omega_{i-1} + \varphi(x(\omega_{i-1}))$ with $\omega_0 = 0$ and $\chi_i = \norm{x(\omega_i)}$. If we can show that
	\begin{multline} \label{eq:lamperti}
		\text{ess } \sup \, \E^0 \left( \chi_{i+1} - \chi_i \, | \, \chi_i = \xi, \chi_j = \xi_j, j \leq i-1 \right)\\
		\leq \frac{1}{2\xi} \text{ess } \inf \ \E^0 \left( (\chi_{i+1} - \chi_i)^2 \, | \, \chi_i = \xi, \chi_j = \xi_j, j \leq i-1 \right) + C \xi^{-1-b}
	\end{multline}
	for some finite constant $C$ and all $\xi$ large enough, where the sup and the inf are taken over $i \geq 1$ and $(\xi_j, 0 \leq j \leq i-1)$, then Theorem~$3.2$ in Lamperti~\cite{Lamperti60:0} will imply that $x$ almost surely visits infinitely often some finite set; since it is irreducible this will prove recurrence. Let $\Fcal_i = \sigma(\chi_j, 0 \leq j \leq i)$ and $\Gcal_i = \sigma(x(\omega_j), 0 \leq j \leq i)$: we have
	\[
		\E^0 \left( \chi_{i+1} - \chi_i \, | \, \Fcal_i \right) = \E^0 \left[ \E^0 \left( \chi_{i+1} - \chi_i \, | \, \Gcal_i \right) \, | \, \Fcal_i \right] = \E^0 \left[ \E^{x(\omega_i)} \left( \norm{x(\varphi(x(0)))} - \norm{x(0)} \right) \, | \, \Fcal_i \right]
	\]
	and hence
	\[
		\E^0 \left( \chi_{i+1} - \chi_i \, | \, \Fcal_i \right) \leq \max_{y: \norm{y} = \chi_i} \E^y \left( \norm{x(\varphi_y)} - \norm{x(0)} \right),
	\]
	writing indifferently $\varphi_y$ or $\varphi(y)$. Similarly,
	\[
		\E^0 \left( (\chi_{i+1} - \chi_i)^2 \, | \, \Fcal_i \right) \geq \inf_{y: \norm{y} = \chi_i} \E^y \left( (\norm{x(\varphi_y)} - \norm{x(0)})^2 \right).
	\]
	
	Remember that $\widetilde \ell = \norm{x(0)}+a-d$: since $\norm{x} \geq \ell$ by Lemma~\ref{lemma:coupling-MM1} and $\ell = \underline {\widetilde \ell} \geq \widetilde \ell$ by definition, we get $(\norm{x(\varphi_y)} - \norm{x(0)})^2 \geq \varphi_y$ in the event $\{ \widetilde \ell(\varphi_y) - \norm{x(0)} \geq \varphi_y^{1/2} \}$ and so
	\[
		\E^y \left( (\norm{x(\varphi_y)} - \norm{x(0)})^2 \right) \geq \E^y \left( (\norm{x(\varphi_y)} - \norm{x(0)})^2 ; \widetilde \ell(\varphi_y) - \norm{x(0)} \geq \varphi_y^{1/2} \right) \geq c \varphi_y
	\]
	with $c = \inf_{t \geq 0} \, \P^0( \widetilde \ell(t) \geq \sqrt t )$ which is strictly positive since $\widetilde \ell(t) / \sqrt t$ under $\P^0$ converges in distribution as $t \to +\infty$ to a normal random variable. Thus to show~\eqref{eq:lamperti} it is enough to show that there exists some finite constant~$C$ such that
\begin{equation}
\label{eq:goal}
\max_{y: \norm{y} = \xi} \E^y \left( \norm{x(\varphi_\xi)} - \norm{x(0)} \right) \leq \frac{C \varphi_\xi}{\xi^{1+b}}
\end{equation}
	for all $\xi$ large enough. Let $\xi \geq 0$ and $y \in \N^K$ such that $\norm{y} = \xi$: integrating~\eqref{eq:sde} over~$\P^y$ and using $\lambda = \mu$ gives
	\[ \E^y \left( \norm{x(\varphi_\xi)} - \norm{x(0)} \right) = \sum_{k=1}^K \mu_k \int_0^{\varphi_\xi} \P^y \left( x_k(u) = 0 \right) du. \]
	
	Let $k \in \{1, \ldots, K\}$: then
	\[ \int_0^{\varphi_\xi} \P^y \left( x_k(u) = 0 \right) du \leq \varphi_\xi \P^y \left( T^\downarrow (x, \xi/2) \leq \varphi_\xi \right) + \int_0^{\varphi_\xi} \P^y \left( x_k(u) = 0, T^\downarrow(x, \xi/2) \geq \varphi_\xi \right) du. \]
	
	Since $\widetilde \ell \leq \norm{x}$ and $\widetilde \ell$ is a symmetric random walk, we have
	\[ \P^y \left( T^\downarrow (x, \xi/2) \leq \varphi_\xi \right) \leq \P^y \left( T^\downarrow (\widetilde \ell, \xi/2) \leq \varphi_\xi \right) = \P^0 \left( T^\uparrow (\widetilde \ell, \xi/2) \leq \varphi_\xi \right) = 2 \P^0 \left( \widetilde \ell(\varphi_\xi) \geq \xi/2 \right) \]
	using the reflection principle to get the last equality. Moreover, writing $\widetilde \ell(t)$ for $t \geq 0$ as $\widetilde \ell(t) = (a(t) - \lambda t) + (\lambda t - d(t))$, using the triangular inequality and the fact that $a$ and $d$ are i.i.d.\ Poisson processes with intensity $\lambda$, we obtain
	\[ \P^0 \left( \widetilde \ell(\varphi_\xi) \geq \xi/2 \right) \leq \P \left( |\Pcal(\lambda \varphi_\xi) - \lambda \varphi_\xi | \geq \xi/4 \right). \]
	
	Extending~\eqref{eq:deviation-P} to upper bound $\P(\Pcal(\lambda \varphi_\xi) \leq \varphi_\xi - \xi/4)$, it can be proved that there exists a finite constant $H > 0$ such that
	\[ \P \left( |\Pcal(\lambda \varphi_\xi) - \lambda \varphi_\xi | \geq \xi/4 \right) \leq \exp \left( -H\xi^2 / \varphi_\xi \right) = \exp \left( -H \sqrt \xi \right). \]
	
	Further, let $\delta > 0$ and $\eta = \underline \pi \delta^2 / 32$ be such that $\delta < \underline \pi$ and $\eta < 2 \varepsilon_0$, and write $\tau = \tau(\eta / (4K))$: then
	\begin{multline*}
		\int_0^{\varphi_\xi} \P^y \left( x_k(u) = 0, T^\downarrow(x, \xi/2) \geq \varphi_\xi \right) du \leq \tau + \int_\tau^{\varphi_\xi} \P^y \left( x_k(u) = 0, T^\downarrow(x, \xi/2) \geq \varphi_\xi \right) du
	\end{multline*}
	and for $u \geq \tau$, the Markov property at time $\tau$ gives
	\begin{multline*}
		\P^y \left( x_k(u) = 0, T^\downarrow(x, \xi/2) \geq \varphi_\xi \right) \leq \P^y \left( \norm{r(\tau) - \pi} \geq \eta \right)\\
		+ \sup_{y'} \, \P^{y'} \left( x_k(u-\tau) = 0, T^\downarrow(x, \xi/2) \geq \varphi_\xi - \tau \right)
	\end{multline*}
	where the supremum is taken over the set $\{y' \in \N^K: \norm{y'} \geq \xi / 2, \varrho(y') \leq \eta\}$. Since $\eta < 2 \varepsilon_0$ and $\tau = \tau(\eta / (4K))$, Lemma~\ref{lemma:bound-open} gives
	\begin{align*}
		\P^y \left( \norm{r(\tau) - \pi} \geq \eta \right) & \leq \P \left( \Pcal(\kappa \tau) \geq \delta \xi / (4K) \right) + 2K \exp \left( - \frac{\delta^2 \xi}{16 K^2} \right)\\
		& \leq \exp\left( - \frac{\kappa \delta \xi}{4 K \kappa \tau} \right) + 2K \exp \left( - \frac{\delta^2 \xi}{16 K^2} \right)
	\end{align*}
	using~\eqref{eq:deviation-P} for the last inequality, together with $h(v) \geq v$ for $v$ large enough (so this inequality holds for $\xi$ large enough). Since $x_k(u) = 0$ implies $\norm{r(u) - \pi} \geq \underline \pi \geq \delta$, we have for any $\tau \leq u \leq \varphi_\xi$ and any $y' \in \N^K$ with $\norm{y'} \geq \xi/2$ and $\varrho(y') \leq \eta$
	\begin{align*}
		\P^{y'} \left( x_k(u-\tau) = 0, T^\downarrow(x, \xi/2) \geq \varphi_\xi - \tau \right) & \leq \P^{y'} \left( T^\uparrow(r - \pi, \delta) \leq u - \tau, T^\downarrow(x, \xi/2) \geq \varphi_\xi - \tau \right)\\
		& \leq \P^{y'} \left( T^\uparrow(r - \pi, \delta) \leq \varphi_\xi \wedge T^\downarrow(x, \xi/2) \right)\\
		& \leq c_2(\delta) \exp\left( \log \varphi_\xi - c_1 (\delta^4 \xi/2 - \log(\xi/2)) \right)
	\end{align*}
	since all the assumptions of Proposition~\ref{prop:control-homogenization} are satisfied (for $\xi$ large enough). Gathering all the previous bounds, we obtain for $\xi$ large enough
	\begin{multline*}
		\E^y \left( \norm{x(\varphi_\xi)} - \norm{x(0)} \right) \leq 4 \mu \varphi_\xi \exp \left( -H \sqrt \xi \right) + \tau \mu + \mu \exp\left( - \frac{\kappa \delta \xi}{4 K \kappa \tau} \right)\\
		+ 2K \mu \exp \left( - \frac{\delta^2 \xi}{16 K^2} \right) + c_2(\delta) \mu \varphi_\xi \exp\left( \log \varphi_\xi - c_1 (\delta^4 \xi/2 - \log(\xi/2)) \right)
	\end{multline*}
	which proves~\eqref{eq:goal} since every term decays exponentially fast, except for the constant term $\tau \mu$ which has to be compared to $\varphi_\xi / \xi^{1+b}$ which goes to infinity.
\end{proof}

\section{Scaling limit}
\label{sec:process}

We assume in this section that the heavy traffic assumption stated in Section~\ref{sub:main-results} holds. In particular, $\rho_n \leq 1$ for every $n \geq 1$ and $\lambda_n \to \lambda > 0$ and $n(1-\rho_n) \to \alpha \geq 0$. 
%
%
%
%
We are interested in the sequence $(X_n, n \geq 1)$ of renormalized processes where $X_n$ for each $n \geq 1$ is given by
\[ X_n(t) = \frac{x_n(n^2 t)}{n}, \ t \geq 0. \]

We define $R_n(t) = r_n(n^2 t)$ which satisfies
$R_n(t) = X_n(t) / \norm{X_n(t)}$ when $\norm{X_n(t)} > 0$, and also
$L_n(t) = n^{-1} \ell_n(n^2 t)$.
In the sequel, we denote by $B$ a Brownian motion with drift
$-\lambda \alpha$, variance $2 \lambda$; $\P^b$ for $b \in \R$ denotes its law started at~$b$. Note that it is known that $L_n \Rightarrow \underline B$ under the heavy-traffic assumption, where from now on $\Rightarrow$ denotes weak convergence. The goal of this section is to prove the following result. 

\begin{theorem}
\label{thm:cv-process}
The sequence of processes $(X_n, n \geq 1)$ under $\P_n^{0}$ converges
weakly as $n$ goes to infinity to $\underline B \pi$ under~$\P^0$.
\end{theorem}

The above theorem is reminiscent of heavy-traffic diffusion limits for
the joint queue length process in various queueing networks,
involving reflected Brownian motion and state space collapse,
see for instance Bramson~\cite{Bramson98:0}, Reiman~\cite{Reiman84:0},
Stolyar~\cite{Stolyar04:0} and Williams~\cite{Williams98:0}.
However, to the best of our knowledge, this is the first result which
shows that mobility of users, rather than scheduling, routing or
load balancing, can act as a mechanism producing state space collapse.

It is straightforward to adapt the proof of Theorem~\ref{thm:cv-process} to handle a more general initial condition.  Let $b > 0$ and assume that $\norm{X_n(0)} \to b$, then it can be proved that:
\begin{itemize}
	\item for any sequence $(\varepsilon_n)$ such that $\varepsilon_n > 0$, $\varepsilon_n \to 0$ and $n^2 \varepsilon_n \to +\infty$, the sequence of shifted processes $(\theta_{\varepsilon_n} X_n)$ converges weakly towards $\underline B \pi$ under~$\P^{b}$;
	\item if $R_n(0) \to \pi$ then $(X_n)$ converges weakly to $\underline B \pi$ under~$\P^{b}$.
\end{itemize}

We see that if $R_n(0) \to \pi' \not = \pi$ then $(X_n)$ converges in the sense of finite-dimensional distributions to a process which is discontinuous at~0 and so cannot converge weakly (at least in the space of c\`adl\`ag functions).

\subsection{Overview of the proof}

Remember from Section~\ref{sub:func-operators} that $e_\varepsilon^\uparrow(X_n) = (\sigma \circ \theta_{T^\uparrow(X_n, \varepsilon)}) (X_n)$ and that $g_\varepsilon(X_n)$ is the left endpoint of the first excursion of~$X_n$ with height larger than $\varepsilon$, equivalently the left endpoint of the excursion of~$X_n$ straddling $T^\uparrow(X_n, \varepsilon)$. To prove Theorem~\ref{thm:cv-process}, we use Theorem~$4$ in Lambert and Simatos~\cite{Lambert12:0}: in particular, Theorem~\ref{thm:cv-process} will be proved if we can show that $g_\varepsilon(X_n) \Rightarrow g_\varepsilon(\underline B \pi)$, $e_\varepsilon^\uparrow(X_n) \Rightarrow e^\uparrow_\varepsilon(\underline B \pi)$ and $(T_0 \circ e_\varepsilon^\uparrow)(X_n) \Rightarrow (T_0 \circ e^\uparrow_\varepsilon) (\underline B \pi)$ for any $\varepsilon > 0$, where $X_n$ is considered under~$\P_n^0$ and $B$ under~$\P^0$. Note that $g_\varepsilon(\underline B \pi) = g_\varepsilon(\underline B)$ and similarly that $(T_0 \circ e^\uparrow_\varepsilon) (\underline B \pi) = (T_0 \circ e^\uparrow_\varepsilon) (\underline B)$.

The convergence of the two sequences $(e_\varepsilon^\uparrow(X_n))$
and $((T \circ e_\varepsilon^\uparrow)(X_n))$ is studied in
Proposition~\ref{prop:conv-e-t}, its proof relies on the following ideas.
First, let $\varpi_n = x_n(T^\uparrow(x_n, \varepsilon n))$:
the Markov property shows that $e_\varepsilon^\uparrow(X_n)$ under~$\P_n^0$
is equal in distribution to $\sigma(X_n)$ under $\P_n^{\varpi_n}$.
In Lemma~\ref{lemma:initial-homogenization} we show that
$\varpi_n \approx \varepsilon \pi$, i.e., the system is with high
probability homogenized at time $T^\uparrow(x_n, \varepsilon n)$.
Then, we want to show that this property lasts on times of order of~$n^2$
that we are interested in.  To do so we exploit the bound of
Proposition~\ref{prop:control-homogenization}, which shows that we will
be able to control the process on times of order of~$n^2$ as long as
there are at least of the order of $\log n$ users in the network.
Thus, this reduces the problem to control $\sigma (X_n)$ started with
$\log n$ users: this is studied in Lemma~\ref{lemma:crude-upper-bound}
where it is shown that $T_0(X_n)$ is at most of order of $\sqrt n$.
Since arrivals and departures are Poisson, on this time scale the total
number of users at most varies by~$\sqrt n$ which is negligible
compared to the space scale~$n$ that we are interested in.

On the other hand, the convergence
$g_\varepsilon(X_n) \Rightarrow g_\varepsilon(\underline B \pi)$,
proved in Proposition~\ref{prop:conv-g}, follows from two arguments:
one that provides a lower bound in terms of a sum of i.i.d.\ terms
related to $(T_0 \circ e_\varepsilon^\uparrow)(X_n)$, whose asymptotic
behavior we will control thanks to the convergences of
$(e_\varepsilon^\uparrow(X_n))$
and $((T_0 \circ e_\varepsilon^\uparrow)(X_n))$; and one that provides
a corresponding upper bound via the coupling with the $M/M/1$ queue.

\subsection{Convergence of $(e_\varepsilon^\uparrow(X_n), n \geq 1)$
and $((T_0 \circ e_\varepsilon^\uparrow)(X_n), n \geq 1)$}
\label{sub:conv}

The convergence of these two sequences is proved in
Proposition~\ref{prop:conv-e-t}. In the following proofs we make repeated use of the constant $\varepsilon_0$ given by Lemma~\ref{lemma:bound-closed} and of the various constants, functions and operators defined in Sections~\ref{sub:mobility} and~\ref{sub:func-operators}. 

\begin{lemma}
\label{lemma:initial-homogenization}
For any $\varepsilon, \delta > 0$,
\[ \lim_{n \to +\infty} \P_n^0 \left( \norm{R_n(T^\uparrow(X_n, \varepsilon)) - \pi} \geq \delta \right) = 0. \]
\end{lemma}

\begin{proof}
	Fix $\varepsilon > 0$, by monotonicity, the result only needs to be proved for $\delta$ small enough; in the sequel we will consider $\delta > 0$ such that $\eta < 2\varepsilon_0$ with $\eta = \underline \pi \delta^2 / 32$. Let $\tau = \tau(\eta / (4K))$, $u_n = n^{1/5}$, $v_n = n^{1/2}$, $T_n = T^\uparrow(x_n, n \varepsilon)$, $T_n' = T^\uparrow(x_n, n \varepsilon - u_n)$ and $E_n$ be the event
	\[ E_n = \left\{ T'_n + \tau \leq T_n \right\} \cap \left\{ \norm{x_n(T'_n + \tau)} \geq n \varepsilon - 2u_n \right\} \cap \left\{ \norm{r_n(T'_n + \tau) - \pi} \leq \eta \right\}. \]
	Using that $T'_n \leq T_n$, the strong Markov property applied at~$T_n'$ gives
	\[
		\P_n^0 \left( E_n^c \right) = \P_n^{\varpi_n'} \left( \left\{ T_n \leq \tau \right\} \cup \left\{ \norm{x_n(\tau)} \leq n\varepsilon - 2u_n \right\} \cup \left\{ \norm{r_n(\tau) - \pi} \geq \eta \right\} \right)
	\]
	where $\varpi_n'$ is equal in distribution to $x_n(T_n')$ under~$\P_n^0$. On the other hand, the strong Markov property applied at the stopping time $T_n' + \tau$ gives
	\[
		\P_n^0 \left( \{ \norm{r_n(T_n)-\pi} \geq \delta \} \cap E_n \right) \leq \max_{y \in \Tcal_n} \, \P_n^y \left( \norm{r_n(T_n) - \pi} \geq \delta \right).
	\]
	with $\Tcal_n = \{ y \in \N^K: n \varepsilon - 2u_n \leq \norm{y} \leq n \varepsilon \text{ and } \varrho(y) \leq \eta \}$. Let $y_n \in \Tcal_n$ that realizes the maximum, so that $\P_n^0( \{ \norm{r_n(T_n)-\pi} \geq \delta \} \cap E_n) \leq \P_n^{y_n}( \norm{r_n(T_n) - \pi} \geq \delta)$. Thus we get the bound
	\begin{multline*}
		\P_n^0 \left( \norm{r_n(T_n)-\pi} \geq \delta \right) \leq \P_n^{\varpi_n'} \left( T_n \leq \tau \text{ or } \norm{x_n(\tau)} \leq n\varepsilon - 2u_n \right) + \P_n^{\varpi_n'} \left( \norm{r_n(\tau) - \pi} \geq \eta \right)\\
		+ \P_n^{y_n} \left( \norm{r_n(T_n) - \pi} \geq \delta \right).
	\end{multline*}
	
	Since $\norm{\varpi_n'} = \lceil n \varepsilon - u_n \rceil$, if under $\P_n^{\varpi_n'}$ we have $T_n \leq \tau$ or $\norm{x_n(\tau)} \leq n\varepsilon - 2u_n$ then we must have at least $u_n$~arrivals or $u_n$~departures in $[0,\tau]$. Since the arrival and departure rates are bounded (by $\kappa$) while $u_n \to +\infty$ we see that the probability of this event vanishes, i.e., $\P_n^{\varpi_n'}(T_n \leq \tau \text{ or } \norm{x_n(\tau)} \leq n \varepsilon - 2u_n) \to 0$. 

	On the other hand, since $\norm{\varpi_n'} = \lceil n \varepsilon - u_n \rceil$, we have
	\[ \P_n^{\varpi_n'} \left( \norm{r_n(\tau) - \pi} \geq \eta \right) \leq \P_n^{y_n'} \left( \norm{r_n(\tau) - \pi} \geq \eta \right) \]
	for some $y_n' \in \N^K$ with $\norm{y_n'} = \lceil n\varepsilon - u_n \rceil$. Since $\eta < 2 \varepsilon_0$ and $\tau = \tau(\eta / (4K))$, Lemma~\ref{lemma:bound-open} shows that $\P_n^{y_n'}(\norm{r_n(\tau) - \pi} \geq \eta) \to 0$. Hence we have
	\[
		\limsup_{n \to +\infty} \ \P_n^0 \left( \norm{r_n(T_n)-\pi} \geq \delta \right) \leq \limsup_{n \to +\infty} \ \P_n^{y_n} \left( \norm{r_n(T_n) - \pi} \geq \delta \right)
	\]
	and we now show that this last upper bound is equal to~0. 
	
	We have
	\[ \P_n^{y_n} \left( \norm{r_n(T_n) - \pi} \geq \delta \right) \leq \P_n^{y_n} \left( T_n \geq v_n \right) + \P_n^{y_n} \left( T^\uparrow(r_n - \pi, \delta) \leq v_n \right) \]
	and by Lemma~\ref{lemma:coupling-MM1},
	\[ \P_n^{y_n} \left( T_n \geq v_n \right) \leq \P_n^{y_n} \left( T^\uparrow(\ell_n, n \varepsilon) \geq v_n \right) \leq \P_n^{y_n''} \left( T^\uparrow(\ell_n, n \varepsilon) \geq v_n \right) \]
	where $y_n'' \in \N^K$ is such that $\norm{y_n''} = \lfloor n \varepsilon - u_n \rfloor$, using stochastic monotonicity of $T^\uparrow(\ell_n, n \varepsilon)$ in the size of the initial condition. Since $\ell_n \geq \widetilde \ell_n$, we get
	\[ \P_n^{y_n} \left( T_n \geq v_n \right) \leq \P_n^{y_n''} \left( T^\uparrow(\widetilde \ell_n, n \varepsilon) \geq v_n \right) = \P_n^{0} \left( T^\uparrow(\widetilde \ell_n, u_n) \geq v_n \right). \]
	
	On the other hand we have
	\[
		\P_n^{y_n} \left( T^\uparrow(r_n - \pi, \delta) \leq v_n \right) \leq \P_n^{y_n} \left( T^\downarrow(\widetilde \ell_n, u_n) \leq v_n \right) + \P_n^{y_n} \left( T^\uparrow(r_n - \pi, \delta) \leq v_n \wedge T^\downarrow(x_n, u_n) \right)
	\]
	where we have used the inequality $T^\downarrow(x_n, u_n) \leq T^\downarrow(\widetilde \ell_n, u_n)$ that stems from Lemma~\ref{lemma:coupling-MM1}. Since $\eta < 2 \varepsilon_0$, $\varrho(y_n) \leq \eta$ and $\norm{y_n} \geq u_n$, all the assumptions of Proposition~\ref{prop:control-homogenization} are satisfied, at least for $n$ large enough. Thus the second term of the right-hand side of the previous display vanishes, and gathering all the previous bounds we see that we are left with
	\begin{multline*}
		\limsup_{n \to +\infty} \ \P_n^0 \left( \norm{r_n(T_n)-\pi} \geq \delta \right) \leq \limsup_{n \to +\infty} \ \P_n^{0} \left( T^\uparrow(\widetilde \ell_n, u_n) \geq v_n \right)\\
		+ \limsup_{n \to +\infty} \P_n^{y_n} \left( T^\downarrow(\widetilde \ell_n, u_n) \leq v_n \right).
	\end{multline*}
	
	Since the sequence of rescaled processes $(\widetilde L_n)$ with $\widetilde L_n(t) = \widetilde \ell_n(n^2 t) / n$ converges in distribution to a Brownian motion, it is not hard to show that the two sequences of random variables $(u_n^{-2} T^\uparrow(\widetilde \ell_n, u_n))$ under~$\P_n^0$ and $(n^{-2} T^\uparrow(\widetilde \ell_n, u_n))$ under $\P_n^{y_n}$ converge weakly to a non-degenerate random variable (actually, hitting times are continuous functionals when the limiting process is the almost sure realization of a Brownian motion, see for instance Proposition VI.$2$.$11$ in Jacod and Shiryaev~\cite{Jacod03:0}). Since $(u_n)^2 \ll v_n \ll n^2$ this finally proves that the right-hand side of the previous display is equal to~0, hence the result.
\end{proof}

\begin{lemma}
\label{lemma:crude-upper-bound}
	Let $\phi_n = \lfloor (\log n)^{2} \rfloor$: then
	\[ \lim_{n \to +\infty} \left( \max_{y: \norm{y} = \phi_n} \P_n^y \left( T_0(x_n) \geq \sqrt n \right) \right) = 0. \]
\end{lemma}

\begin{proof}
	Let $M_n \in \N$ be such that $2^{M_n-1} < \phi_n \leq 2^{M_n}$. By monotonicity of $T_0$ in the size of the initial state, we have
	\[ \max_{y: \norm{y} = \phi_n} \ \P_n^y \left( T_0 \geq \sqrt n \right) \leq \max_{y \in \Tcal(M_n)} \ \P_n^y \left( T_0 \geq \sqrt n \right) \]
	where from now on $\Tcal(m) = \{ y \in \N^K: \norm{y} = 2^{m} \}$ and we omit the dependency of the functional operators when they are applied at $x_n$, so that $T_0 = T_0(x_n)$. Define $S(f) = T^\downarrow(f, \norm{f(0)}/2)$ and $\varphi_{n,m} = \exp(M_n + 2^{(M_n-m)/4})$, and note that since $\rho_n \leq 1$ by assumption, $S$ is almost surely finite since $x_n$ is recurrent by Proposition~\ref{prop:critical-case}. The relation $T_0 = S + T_0 \circ \theta_S$ gives for any $n \geq 1$ and $y \in \N^K$
	\begin{align*}
		\P_n^y \left( T_0 \geq \sqrt n \right) & = \P_n^y \left( T_0 \geq \sqrt n, S \geq \varphi_{n,0} \right) + \P_n^y \left( T_0 \geq \sqrt n, S \leq \varphi_{n,0} \right)\\
		& \leq \P_n^y \left( S \geq \varphi_{n,0} \right) + \P_n^y \left( T_0 \circ \theta_S \geq \sqrt n - \varphi_{n,0} \right).
	\end{align*}
	
	Since moreover $\P_n^y(x_n(S) \in \Tcal(M_n-1)) = 1$ for $y \in \Tcal(M_n)$, the strong Markov property at $S$ gives
	\[ \max_{y \in \Tcal(M_n)} \P_n^y \left( T_0 \geq \sqrt n \right) \leq \max_{y \in \Tcal(M_n)} \P_n^y \left( S \geq \varphi_{n,0} \right) + \max_{y \in \Tcal(M_{n}-1)} \P_n^y \left( T_0 \geq \sqrt n - \varphi_{n,0} \right).
	\]
	
	Iterating $M_n+1$ times, we obtain
	\begin{multline*}
		\max_{y \in \Tcal(M_n)} \P_n^y \left( T_0 \geq \sqrt n \right) \leq \sum_{m=0}^{M_n} \ \max_{y \in \Tcal(M_{n}-m)} \P_n^y \left( S \geq \varphi_{n,m} \right)\\
		+ \max_{y \in \Tcal(0)} \ \P_n^y \left( T_0 \geq \sqrt n - \varphi_{n,0} - \cdots - \varphi_{n,M_n} \right).
	\end{multline*}
	Now $\Tcal(0) = \{0\}$ and $T_0$ under~$\P_n^0$ is equal to~0. Since in addition $2^{M_n} \leq 2 \phi_n \leq 2 (\log n)^{2}$, we obtain $2^{M_n/4} \leq 2^{1/4} (\log n)^{1/2}$ and so
	\[ \varphi_{n,0} + \cdots + \varphi_{n,M_n} \leq (M_n+1) \exp\left( M_n + 2^{M_n/4} \right) < \exp \left( (1/2) \log n \right) =  \sqrt n \]
	for $n$ large enough. For those $n$, we get
	\[ \max_{y \in \Tcal(0)} \ \P_n^y \left( T_0 \geq \sqrt n - \varphi_{n,0} - \cdots - \varphi_{n,M_n} \right) = 0 \]
	and finally, we obtain after a change of variables
	\begin{equation} \label{eq:back}
		\max_{y \in \Tcal(M_n)} \P_n^y \left( T_0 \geq \sqrt n \right) \leq \sum_{m=0}^{M_n} \ \max_{y \in \Tcal(m)} \ \P_n^y \left( S \geq \exp(M_n + 2^{m/4}) \right).
	\end{equation}
	
	Let us justify that for fixed $m \geq 0$ and $y \in \N^K$ we have
	\begin{equation} \label{eq:inter}
		\lim_{n \to +\infty} \P_n^y \left( S \geq \exp(M_n + 2^{m/4}) \right) = 0.
	\end{equation}
	
	First, let $(u(n))$ be a subsequence such that
	\[ \lim_{n \to +\infty} \P_{u(n)}^y \left( S \geq \exp(M_{u(n)} + 2^{m/4}) \right) = \limsup_{n \to +\infty} \P_n^{y} \left( S \geq \exp(M_n + 2^{m/4}) \right). \]

	Since the $[0,\infty)^{2K}$-valued sequence $((\lambda_{u(n),k}, \mu_{u(n),k}, 1 \leq k \leq K), n \geq 1)$ lives in a compact set as a consequence of the heavy-traffic assumption, we can find a subsequence $(v(n))$ of $(u(n))$ and $\lambda_{\infty}, \mu_{\infty} \in [0,\infty)^{K}$ such that $\lambda_{v(n),k} \to \lambda_{\infty, k}$ and $ \mu_{v(n),k} \to \mu_{\infty, k}$ for each $k = 1, \ldots, K$. Because of the heavy-traffic assumption, we have $\norm{\lambda_\infty} = \norm{\mu_\infty}$.
	
It is then not hard to see that the sequence $(x_{v(n)})$ under~$\P_n^{y}$
converges weakly to~$x_\infty$, where $x_\infty$ is the Markov process
with $x_\infty(0) = y$ and generator~$\Omega_\infty$ defined similarly
as $\Omega_n$ but with $\lambda_{n,k}$ and $\mu_{n,k}$ replaced by
$\lambda_{\infty, k}$ and $\mu_{\infty, k}$, respectively.
Since $x_\infty$ lives in $\N^K$ and is piecewise constant, it is not
hard to prove that $S$ is a continuous functional at $x_\infty$
and so the continuous-mapping theorem implies the weak convergence of
the sequence $(S(x_{v(n)}))$ towards $S(x_\infty)$.
In particular,
	\[ \lim_{n \to +\infty} \P_{u(n)}^y \left( S \geq \exp(M_{u(n)} + 2^{m/4}) \right) = \P \left( S(x_\infty) = +\infty \right). \]
	
	Since $\norm{\lambda_\infty} = \norm{\mu_\infty}$, $x_\infty$ is recurrent by Proposition~\ref{prop:critical-case} and so $S(x_\infty)$ is finite almost surely which proves~\eqref{eq:inter}. Since for each fixed $m \geq 1$ the set $\Tcal(m)$ is finite, combining~\eqref{eq:back} and~\eqref{eq:inter} we obtain for any $M \geq 0$
	\[
		\limsup_{n \to +\infty} \left( \max_{y \in \Tcal(M_n)} \P_n^y \left( T_0 \geq \sqrt n \right) \right) \leq \limsup_{n \to +\infty} \left( \sum_{m=M}^{M_n} \max_{y \in \Tcal(m)} \ \P_n^y\left(S \geq \exp(2^{m/4}) \right) \right) \leq \sum_{m \geq M} U_m
	\]
	where
	\[ U_{m} = \sup_{n \geq 1} \left( \max_{y \in \Tcal(m)} \ \P_n^y\left(S \geq \exp(2^{m/4})\right) \right). \]
	
	Thus if we can prove that the series $(U_m)$ is summable, letting $M \to +\infty$ in the previous inequality will show the result.

	For the rest of the proof, fix any $\delta > 0$ such that $\delta \leq \underline \pi$ and $\eta < 2 \varepsilon_0$ where $\eta = \underline \pi \delta^2 / 32$. Let in addition $\tau = \tau(\eta / (4K))$. Then for any $y \in \N^K$, we have
	\begin{multline*}
		\P_n^y \left( S \geq \exp(2^{m/4}) \right) \leq \P_n^y \left( S \geq \exp(2^{m/4}), \norm{r_n(\tau) - \pi} \leq \eta, \left| \norm{x_n(\tau)} - \norm{x_n(0)} \right| \leq m \right)\\
		+ \P_n^y \left( \norm{r_n(\tau) - \pi} \geq \eta \right) + \P_n^y \left( \left| \norm{x_n(\tau)} - \norm{x_n(0)} \right| \geq m \right).
	\end{multline*}
	
	Invoking the Markov property at time~$\tau$ for the first term and~\eqref{eq:coupling-inequalities} together with Lemma~\ref{lemma:coupling-closed-system} for the two last ones, we get
	\begin{multline*}
		\max_{y \in \Tcal(m)} \ \P_n^y \left( S \geq \exp(2^{m/4}) \right) \leq \max_{y \in \Tcal'(m, \eta)} \P_n^y \left( S \geq \exp(2^{m/4}) - \tau \right)\\
		+ \max_{y \in \Tcal(m)} \P_n^y \left( \norm{r_n(\tau) - \pi} \geq \eta \right) + \P \left( \Pcal(\kappa \tau) \geq m \right)
	\end{multline*}
	where $\Tcal'(m,\delta) = \{ y \in \N^K: |\norm{y} - 2^m| \leq m \text{ and } \varrho(y) \leq \eta \}$. The last term of the above upper bound defines a summable series since a Poisson random variable has a finite mean; the second term also by Lemma~\ref{lemma:bound-open} since $\eta < 2 \varepsilon_0$. It remains to control the first term. 

	So consider $y \in \Tcal'(m, \delta)$ and let $v_m = \exp(2^{m/4}) - \tau$ and $\psi_m = 2^{m-1}-m/2$. Since $\norm{y} \geq 2^m-m$ we have $S = T^\downarrow(x_n, \norm{y}/2) \leq T^\downarrow(x_n, \psi_m)$ under~$\P_n^y$ and we obtain
	\begin{multline*}
		\P_n^y \left( S \geq v_m \right) \leq \P_n^y \left( T^\downarrow(x_n, \psi_m) \geq v_m, T^\uparrow(r_n-\pi, \delta) \geq v_m \right)\\
		+ \P_n^y \left( T^\downarrow(x_n, \psi_m) \geq v_m, T^\uparrow(r_n-\pi, \delta) \leq v_m \right).
	\end{multline*}
	
	Since $\delta \leq \underline \pi$, we have $\norm{x_n(t)} = \ell_n(t)$ for all $t \leq T^\uparrow(r_n - \pi, \delta)$ by~\eqref{eq:lower-bound-tilde} and so
	\[
		\P_n^y \left( T^\downarrow(x_n, \psi_m) \geq v_m, T^\uparrow(r_n-\pi, \delta) \geq v_m \right) = \P_n^y \left( T^\downarrow(\ell_n, \psi_m) \geq v_m, T^\uparrow(r_n-\pi, \delta) \geq v_m \right)
	\]
	which implies
	\begin{multline*}
		\max_{y \in \Tcal'(m, \eta)} \P_n^y \left( S \geq v_m \right) \leq \max_{y \in \Tcal'(m, \eta)} \P_n^y \left( T^\downarrow(\ell_n, \psi_m) \geq v_m \right)\\
		+ \max_{y \in \Tcal'(m, \eta)} \P_n^y \left( T^\uparrow(r_n-\pi, \delta) \leq v_m \wedge T^\downarrow(x_n, \psi_m) \right).
	\end{multline*}
	
	As for the second term, it is easily checked that all the assumptions of Proposition~\ref{prop:control-homogenization} are satisfied, at least for $m$ large enough, since $\delta$ and $\eta < 2\varepsilon_0$ are fixed and both $v_m$ and $2^{m-1}$ grow without bounds with~$m$. Proposition~\ref{prop:control-homogenization} provides a bound uniform in $y \in \Tcal'(m,\delta)$ and $n \geq 1$ which defines a summable series in~$m$. Hence to complete the proof it remains to show that
\begin{equation}
\label{eq:sum}
\sum_{m \geq M} \sup_{n \geq 1} \left\{ \P_n^{y_m} \left( T^\downarrow(\ell_n, \psi_m) \geq v_m \right) \right\} < +\infty
\end{equation}
with $y_m = (2^m+m) e_1$, using monotonicity of $T^\downarrow(\ell_n, \psi_m)$ in the size of the initial state and the fact that $\ell_n$ only depends on the initial state through its size. Note that
	\[
		\P_n^{y_m} \left( T^\downarrow(\ell_n, \psi_m) \geq v_m \right) = \P_n^{y_m} \left( T^\downarrow(\widetilde \ell_n, \psi_m) \geq v_m \right) = \P_n^{0} \left( T^\downarrow(\widetilde \ell_n, \psi'_m) \geq v_m \right)
	\]
	with $\psi_m' = \psi_m - 2^m - m = 2^{m-1} - 2^m - 3m/2$.
It is well-known that $T^\downarrow(\widetilde \ell_n, \psi'_m)$
under~$\P_n^0$ scales like $(\psi'_m)^2 \approx 2^{2m}$ for large~$m$
and~$n$, which is negligible compared to $v_m \approx \exp(2^{m/4})$.
In the case $\rho_n = 1$, we can use the reflection principle
and establish an exponential upper bound as we have done in the proof
of Proposition~\ref{prop:critical-case}.
When $\rho_n < 1$ we can couple $\widetilde \ell_n$ with a critical
random walk~$\ell'_n$ such that $\ell'_n \geq \widetilde \ell_n$.
In particular we have $T^\downarrow(\widetilde \ell_n, \psi_m) \leq
T^\downarrow(\ell'_n, \psi_m)$
and so $\P_n^{y_m}( T^\downarrow(\widetilde \ell_n, \psi_m) \geq v_m ) \leq
\P_n^{y_m}( T^\downarrow(\ell'_n, \psi_m) \geq v_m)$,
where this last term obeys to an exponential upper bound.
This proves~\eqref{eq:sum} and completes the proof.
\end{proof}

\begin{prop}
\label{prop:conv-e-t}
	For any $\varepsilon > 0$, the sequences $(e_\varepsilon^\uparrow(X_n), n \geq 1)$ and $((T_0 \circ e_\varepsilon^\uparrow)(X_n), n \geq 1)$ under~$\P_n^0$ converge weakly to $e_\varepsilon^\uparrow(\underline B \pi)$ and $(T_0 \circ e_\varepsilon^\uparrow)(\underline B)$ under~$\P^0$, respectively.
\end{prop}

\begin{proof}
In the rest of the proof fix any $\varepsilon > 0$ and let $\varpi_n$
be the law of $x_n(T^\uparrow(x_n, \varepsilon))$ under~$\P_n^0$.
In view of the strong Markov property, the convergence properties
claimed are equivalent to the convergence of the two sequences
$(\sigma(X_n), n \geq 1)$ and $(T_0(X_n)$, $n \geq 1)$
under~$\P_n^{\varpi_n}$ towards $\sigma(\underline B \pi)$
and $T_0(\underline B)$ under~$\P^\varepsilon$, respectively.
	
	We know that $(\sigma(L_n \pi))$ and $(T_0(L_n))$ under $\P_n^{\varpi_n}$ converge weakly to $\sigma(\underline B \pi)$ and $T_0(\underline B)$ under~$\P^\varepsilon$, respectively, and we want to transfer this result to~$X_n$ using Theorem~$3.1$ in Billingsley~\cite{Billingsley99:0}, sometimes referred to as a ``convergence-together'' result. Thus we only have to prove that for any $\beta > 0$,
	\[ \lim_{n \to +\infty} \P_n^{\varpi_n} \left( \sup_{t \geq 0} \norm{\sigma(X_n)(t) - \sigma(L_n \pi)(t)} \geq \beta \right) = \lim_{n \to +\infty} \P_n^{\varpi_n} \left( |T_0(X_n) - T_0(L_n)| \geq \beta \right) = 0. \]
	
	Define $\phi_n = \lfloor (\log n)^2 \rfloor$ as well as $T_{n,X} = T^\downarrow(X_n, n^{-1} \phi_n)$ and $T_{n,L} = T^\downarrow(L_n, n^{-1} \phi_n)$: then
	\[ T_0(X_n) = T_{n,X} + (T_0 \circ \theta_{T_{n,X}}) (X_n) \ \text{ and } \ T_0(L_n) = T_{n,L} + (T_0 \circ \theta_{T_{n,L}}) (L_n) \]
	so that
	\begin{multline*}
		\P_n^{\varpi_n} \left( |T_0(X_n) - T_0(L_n)| \geq \beta \right) \leq \P_n^{\varpi_n} \left( |T_{n,X} - T_{n,L}| > 0 \right)\\
		+ \P_n^{\varpi_n} \left( (T_0 \circ \theta_{T_{n,X}})(X_n) \geq \beta / 3 \right) + \P_n^{\varpi_n} \left( (T_0 \circ \theta_{T_{n,L}})(L_n) \geq \beta / 3 \right).
	\end{multline*}
	Note that Lemma~\ref{lemma:coupling-MM1} implies that $\P_n^{\varpi_n}(|T_{n,X} - T_{n,L}| > 0) = \P_n^{\varpi_n}(T_{n,X} > T_{n,L})$. Together with the strong Markov property at $T_{n,X}$ and $T_{n,L}$, this gives
	\begin{multline*}
		\P_n^{\varpi_n} \left( |T_0(X_n) - T_0(L_n)| \geq \beta \right) \leq \P_n^{\varpi_n} \left( T_{n,X} > T_{n,L} \right) + \max_{y: \norm{y} = \phi_n} \P_n^y \left( T_0(x_n) \geq n^2 \beta / 3 \right)\\
		+ \max_{y: \norm{y} = \phi_n} \P_n^y \left( T_0(\ell_n) \geq  n^2 \beta / 3 \right).
	\end{multline*}
	
Since the sequence $(\phi_n^{-2} T_0(\ell_n))$ under $\P_n^{\phi_n e_1}$ converges in distribution to a non-degenerate random variable, the last term goes to~$0$ since $n^2 \gg \phi_n^{2}$. The second
term goes to~$0$ by Lemma~\ref{lemma:crude-upper-bound}, and so it remains to control the first term.  In the rest of the proof, let
$\delta > 0$ such that $\delta \leq \underline \pi$ and $\eta = \underline \pi \delta^2 / 32 < 2 \varepsilon_0$: we have
	\[
		\P_n^{\varpi_n} \left( T_{n,X} > T_{n,L} \right) \leq \P_n^{\varpi_n} \left( \norm{R_n(0) - \pi} \geq \eta \right) + \max_{y \in \Tcal_n} \ \P_n^y \left( T_{n,X} > T_{n,L} \right)
	\]
	with $\Tcal_n = \{ y \in \N^K: \norm{y} = \lfloor n \varepsilon \rfloor \text{ and } \varrho(y) \leq \eta \}$. The first term goes to~0 by Lemma~\ref{lemma:initial-homogenization} and so we have at that point
	\[
		\limsup_{n \to +\infty} \ \P_n^{\varpi_n} \left( |T_0(X_n) - T_0(L_n)| \geq \beta \right) \leq \limsup_{n \to +\infty} \left( \max_{y \in \Tcal_n} \ \P_n^y \left( T_{n,X} > T_{n,L} \right) \right).
	\]
	
	Thanks to Lemma~\ref{lemma:coupling-MM1}, one sees that $T_{n,X} > T_{n,L}$ implies that there was a time $t < T_{n,L}$ such that $x_{n,k}(t) = 0$ for some~$k$. At that time we have $\norm{r_n(t) - \pi} \geq \underline \pi \geq \delta$ and so
	\[ \P_n^y \left( T_{n,X} > T_{n,L} \right) \leq \P_n^y \left( T^\uparrow(r_n - \pi, \delta) \leq T_{n,L} \right). \]
	
	Further we have
	\[ \P_n^y \left( T^\uparrow(r_n - \pi, \delta) \leq T_{n,L} \right) \leq \P_n^y \left( n \leq T_{n,L} \right) + \P_n^y \left( T^\uparrow(r_n - \pi, \delta) \leq n \wedge T_{n,L} \right). \]
	
	The first term goes to~0 uniformly in~$y$ with $\norm{y} = \lfloor \varepsilon n \rfloor$ since the sequence $(T_{n,L})$ under~$\P_n^y$ with $y \in \Tcal_n$ converges in distribution, while the second term can be rewritten as
	\[ \P_n^y \left( T^\uparrow(r_n - \pi, \underline \pi) \leq n \wedge T_{n,L} \right) = \P_n^y \left( T^\uparrow(r_n - \pi, \delta) \leq n^{3} \wedge T^\downarrow(\ell_n, \phi_n) \right) \]
	which goes to~0 uniformly in $y \in \Tcal_n$ by Proposition~\ref{prop:control-homogenization} (using $T^\downarrow(\ell_n, \phi_n) \leq T^\downarrow(x_n, \phi_n)$). This proves the result on $T_0(X_n)$, the result on $\sigma(X_n)$ follows along the same lines but at the expense of more technical details. Lemma~\ref{lemma:coupling-MM1} and~\eqref{eq:lower-bound-tilde} imply that $T_{n,X} = T_{n,L}$ and $L_n(T_{n,L}) = \norm{X_n(T_{n,X})} = n^{-1} \phi_n$ in the event $\{ T^\uparrow(R_n - \pi, \delta) \geq T_{n,L} \}$, so that the strong Markov property gives
	\begin{multline*}
		\P_n^{\varpi_n} \left( \sup_{t \geq 0} \norm{\sigma(X_n)(t) - \sigma(L_n \pi)(t)} \geq \beta \right) \leq \P_n^{\varpi_n} \left( T^\uparrow(R_n - \pi, \delta) \geq T_{n,L} \right)\\
		+ \P_n^{\varpi_n} \left( \sup_{0 \leq t \leq T_{n,L}} \norm{X_n(t) - L_n(t) \pi} \geq \beta/3 \right) + \max_{y \in \Phi_n} \ \P_n^y \left( \sup_{0 \leq t \leq T_0(X_n)} \norm{X_n(t)} \geq \beta/3 \right)\\
		+ \max_{y \in \Phi_n} \ \P_n^y \left( \sup_{0 \leq t \leq T_0(L_n)} L_n(t) \geq \beta/3 \right)
	\end{multline*}
	where $\Phi_n = \{ y \in \N^K: \norm{y} = \phi_n \}$. We have already proved earlier in the proof that the first term vanishes. The last term can be seen to vanish invoking the convergence of~$(L_n)$ towards a Brownian motion. The third term vanishes because for any $y \in \Phi_n$,
	\begin{align*}
		\P_n^y \left( \sup_{0 \leq t \leq T_0(X_n)} \norm{X_n(t)} \geq \beta/3 \right) & = \P_n^y \left( \sup_{0 \leq t \leq T_0(x_n)} \norm{x_n(t)} \geq n\beta/3 \right)\\
		& \leq \P_n^y \left( T_0(x_n) \geq \sqrt n \right) + \P_n^y \left( a_n(\sqrt n) \geq n \beta / 3 - \phi_n \right)
	\end{align*}
	and both terms go to~0 uniformly in $y \in \Phi_n$, the first one using Lemma~\ref{lemma:crude-upper-bound} and the second one using Markov inequality. It remains to show that the second term also vanishes.

	Let $S_n = \sup_{[0, T_{n,L}]} L_n$: then for any $0 \leq t \leq T_{n,L}$, one has
	\begin{align*}
		\norm{X_n(t) - L_n(t) \pi} & \stackrel{(1)}{\leq} \norm{X_n(t)} \norm{R_n(t) - \pi} + \norm{\pi} \left| \norm{X_n(t)} - L_n(t) \right|\\
		& \stackrel{(2)}{\leq} L_n(t) \norm{R_n(t) - \pi} + 2 \left| \norm{X_n(t)} - L_n(t) \right|\\
		& \stackrel{(3)}{\leq} S_n \norm{R_n(t) - \pi} + 2 \left| \norm{X_n(t)} - L_n(t) \right|
	\end{align*}
	where $(1)$ follows by adding and subtracting $\pi \norm{X_n(t)}$ and using the triangular inequality, $(2)$ follows from $\norm{X_n(t)} \leq L_n(t) + |\norm{X_n(t)} - L_n(t)|$ together with $\norm{R_n(t) - \pi} \leq 1$ and $\norm \pi = 1$ and $(3)$ is by definition of $S_n$, since $t \leq T_{n,L}$. Using this upper bound together with standard manipulations, we get the following upper bound, valid for any $s > 0$:
	\begin{multline*}
		\P_n^{\varpi_n} \left( \sup_{0 \leq t \leq T_{n,L}} \norm{X_n(t) - \pi L_n(t)} \geq \beta/3 \right) \leq \P_n^{\varpi_n} \left( \sup_{0 \leq t \leq T_{n,L}} \norm{R_n(t) - \pi} \geq \beta / (6 s) \right)\\
		+ \P_n^{\varpi_n}(S_n \geq s) + \P_n^{\varpi_n} \left( \sup_{0 \leq t \leq T_{n,L}} \left| \norm{X_n(t)} - L_n(t) \right| > 0 \right).
	\end{multline*}
	
	Similarly as before the first and last terms go to~0, so we are left with
	\[
		\limsup_{n \to +\infty} \P_n^{\varpi_n} \left( \sup_{0 \leq t \leq T_{n,L}} \norm{X_n(t) - \pi L_n(t)} \geq \beta \right) \leq \limsup_{n \to +\infty} \P_n^{\varpi_n}(S_n > s).
	\]
	
	Letting $s \to +\infty$ completes the proof, since the sequence $(S_n)$ under $\P_n^{\varpi_n}$ converges weakly to $\sup_{[0, T_0(B)]} B$~under $\P^\varepsilon$.
\end{proof}

\subsection{Convergence of $(g_\varepsilon(X_n), n \geq 1)$}
\label{sub:conv-g}

To complete the proof of Theorem~\ref{thm:cv-process} based on Theorem~$4$
in Lambert and Simatos~\cite{Lambert12:0} it remains to be shown that
$g_\varepsilon(X_n) \Rightarrow g_\varepsilon(\underline B)$.

\begin{prop}
\label{prop:conv-g}
	For any $\varepsilon > 0$, the sequence $(g_\varepsilon(X_n))$ under~$\P_n^0$ converges weakly to $g_\varepsilon(\underline B)$ under~$\P^0$.
\end{prop}

\begin{proof}
	Since $\norm{X_n} \geq L_n$ by Lemma~\ref{lemma:coupling-MM1}, it is clear that $T^\uparrow(X_n, \varepsilon) \leq T^\uparrow(L_n, \varepsilon)$, and hence going back in time and using again $\norm{X_n} \geq L_n$ we see that $g_\varepsilon(X_n) \leq g_\varepsilon(L_n)$. On the other hand, since $L_n \Rightarrow \underline B$ it is not difficult to show that $g_\varepsilon(L_n) \Rightarrow g_\varepsilon(\underline B)$, see for instance Lambert and Simatos~\cite{Lambert11:0} where similar computations are carried out. This proves that the sequence $(g_\varepsilon(X_n))$ is tight and that any accumulation point is stochastically dominated by $g_\varepsilon(\underline B)$. We now derive a corresponding lower bound which will conclude the proof. \\
	
	Let $\delta > 0$ and denote by $E^\uparrow_{n,k,\delta}$ the $k$th excursion of~$X_n$ that reaches level $\delta$ shifted at the first time it reaches this value: formally, we have
	\[ E^\uparrow_{n,1,\delta} = e_\delta^\uparrow(X_n) \ \text{ and } \ E^\uparrow_{n,k+1,\delta} = \left( e_\delta^\uparrow \circ \theta_{d_{n,k,\delta}} \right) (X_n) \]
	with $d_{n,k,\delta}$ the right endpoint of the excursion of $X_n$ corresponding to $E_{n,k,\delta}$. Let $N_{n,\delta,\varepsilon}$ be the number of excursions of~$X_n$ that reach level $\delta$ and not level $\varepsilon$ before the first excursion of~$X_n$ to reach level~$\varepsilon$: it satisfies
	\[ N_{n,\delta,\varepsilon} + 1 = \inf \left \{ k \geq 1: \sup_{t \geq 0} \norm{E^\uparrow_{n,k,\delta}(t)} \geq \varepsilon \right\}. \]
	
	Then $N_{n,\delta,\varepsilon}$ is a geometric random variable with parameter $p_{n,\delta,\varepsilon}$ given by $p_{n,\delta,\varepsilon} = \P_n^0(\sup \norm{e_\delta^\uparrow(X_n)} < \varepsilon)$, the $(E^\uparrow_{n,k,\delta}, 1 \leq k \leq N_{n, \delta, \varepsilon})$ are i.i.d., independent of $N_{n,\delta,\varepsilon}$ and with common distribution $e_\delta^\uparrow(X_n)$ conditioned on $\{ \sup \norm{e_\delta^\uparrow(X_n)} < \varepsilon \}$, and we have
	\[ g_\varepsilon(X_n) \geq \sum_{k=1}^{N_{n,\delta,\varepsilon}} T_0\left( E^\uparrow_{n,k,\delta} \right) \]
	with the convention $\sum_1^0 = 0$. This last inequality implies for any $s > 0$
	\[
		\E_n^0 \left( e^{-s g_\varepsilon(X_n)} \right) \leq \E \left[ \left\{ \E_n^0 \left( e^{-s (T_0 \circ e_\delta^\uparrow)(X_n)} \, | \, \sup \norm{e_\delta^\uparrow(X_n)} < \varepsilon \right) \right\}^{N_{n,\delta,\varepsilon}} \right]
	\]
	and it can be computed that this last upper bound is equal to
	\[
		\frac{1-p_{n,\delta,\varepsilon}}{1-p_{n,\delta,\varepsilon}\E_n^0 \left( e^{-s (T_0 \circ e_\delta^\uparrow)(X_n)} \, | \, \sup \norm{e_\delta^\uparrow(X_n)} < \varepsilon \right)}
		= \frac{\P_n^0(\sup \norm{e_\delta^\uparrow(X_n)} \geq \varepsilon)}{1-\E_n^0 \left( e^{-s (T_0 \circ e_\delta^\uparrow)(X_n)} ; \sup \norm{e_\delta^\uparrow(X_n)} < \varepsilon \right)}.
	\]
	
	Proposition~\ref{prop:conv-e-t} and the continuous-mapping theorem give
	\[ \lim_{n \to +\infty} \P_n^0\left(\sup \norm{e_\delta^\uparrow(X_n)} \geq \varepsilon\right) = \P^0\left(\sup e_\delta^\uparrow(\underline B) \geq \varepsilon\right). \]
	
	On the other hand, the two convergences of Proposition~\ref{prop:conv-e-t} can be shown to hold jointly, see for instance Lambert and Simatos~\cite{Lambert12:0}, and so the continuous-mapping theorem gives
	\[ \lim_{n \to +\infty} \E_n^0 \left( e^{-s (T_0 \circ e_\delta^\uparrow)(X_n)} ; \sup \norm{e_\delta^\uparrow(X_n)} < \varepsilon \right) = \E^0 \left( e^{-s (T_0 \circ e_\delta^\uparrow)(\underline B)} ; \sup e_\delta^\uparrow(\underline B) < \varepsilon \right). \]
	
	Thus for any $\delta > 0$ we have
	\[ \limsup_{n \to +\infty} \E_n^0 \left( e^{-s g_\varepsilon(X_n)} \right) \leq \frac{\P^0(\sup \norm{e_\delta^\uparrow(\underline B)} \geq \varepsilon)}{1-\E^0 \left( e^{-s (T_0 \circ e_\delta^\uparrow)(\underline B)} ; \sup \norm{e_\delta^\uparrow(\underline B)} < \varepsilon \right)}. \]
	
	Using standard arguments from excursion theory, the above upper bound is seen to converge towards $\E^0(e^{-s g_\varepsilon(\underline B)})$ as $\delta \to 0$, and so we finally get
	\[ \limsup_{n \to +\infty} \E_n^0 \left( e^{-s g_\varepsilon(X_n)} \right) \leq \E^0 \left( e^{-s g_\varepsilon(\underline B)} \right). \]
	
	This implies that any accumulation point of the tight sequence $(g_\varepsilon(X_n))$ is stochastically lower bounded by $g_\varepsilon(\underline B)$, and since a corresponding stochastic upper bound holds this gives the result.
\end{proof}

\section{Convergence of the stationary distributions}
\label{sec:conv-distr}

Throughout this section the heavy-traffic assumption of Section~\ref{sub:main-results} continues to be
in force, and we moreover assume that $\alpha > 0$, so that $\rho_n < 1$ for $n$ large enough.
Recall that in this case $x_n$ is positive-recurrent,
see~\cite{Ganesh10:0, Simatos10:0}, and that $\nu_n$ denotes its
stationary distribution.
This section is devoted to proving the following result, where we write
similarly as in the previous section $X_n(0) = n^{-1} x_n(0)$.

\begin{theorem}
\label{thm:ht-stationary-distribution}
The sequence $(X_n(0), n \geq 1)$ under $\P_n^{\nu_n}$ converges weakly
as $n$ goes to infinity to $E \pi$ where $E$ is an exponential random
variable with parameter~$\alpha$, and all higher moments converge as well,
i.e., $\E_n^{\nu_n}(\norm{X_n(0)}^r) \to r!/\alpha^r$ for all integer $r \geq 0$.
\end{theorem}

Since the exponential random variable~$E$ in the above theorem has the
stationary distribution of the reflected Brownian motion~$\underline B$ introduced
in Theorem~\ref{thm:cv-process}, we may observe that the heavy-traffic
characteristics are preserved under an interchange of limits.
While such an interchange of limits tends to apply in most specific
cases, there do not appear to be any general guarantees for that. The proof of Theorem~\ref{thm:ht-stationary-distribution} relies on
the following estimate, and we will in particular make use of the constants $\underline \pi$ and $\varepsilon_0$ defined in Section~\ref{sub:mobility} and Lemma~\ref{lemma:bound-closed}, respectively.

\begin{lemma}
\label{lemma:constant}
There exist two constants $c, c' \in (0,\infty)$ such that for every $n \geq 1$, every $q \geq 0$ and every $1 \leq k \leq K$,
\[ \P_n^{\nu_n} \left( \norm{x_n(0)} \geq q, x_{n,k}(0) = 0 \right) \leq c e^{-c' q}. \]
\end{lemma}

\begin{proof}
	Fix $n \geq 1$, $q \geq 0$ and $1 \leq k \leq K$. Recall the constant $\varepsilon_0$ of Lemma~\ref{lemma:bound-closed} and let $0 < \varepsilon < 1$ be any number such that
	\[ 1 - \frac{\varepsilon_0}{\underline \pi} < \frac{2\varepsilon}{\underline \pi (1-\varepsilon)} < 1. \]
	
	Let $\eta = 2\varepsilon / (\underline \pi(1-\varepsilon))$, so that $0 < \eta < 1$ and $(1-\eta) \underline \pi < \varepsilon_0$, and $\tau = \tau((1-\eta) \underline \pi / (2K))$. Then for any $y \in \N^K$, one has
	\begin{multline*}
		\P_n^y \left( \norm{x_n(\tau)} \geq q, x_{n,k}(\tau) = 0 \right) \leq \P_n^y \left( \norm{x_n(\tau)} \geq q, r_{n,k}(\tau) \leq \varepsilon \pi_k \right)\\
		\leq \P_n^y \left( \norm{x_n(\tau)} \geq q, x_{n,k}(\tau) \leq \varepsilon \pi_k \norm{x_n(\tau)}, a_n(\tau) \vee d_n(\tau) \leq \varepsilon q \right)\\
		+ \P_n^y \left( a_n(\tau) \vee d_n(\tau) \geq \varepsilon q \right).
	\end{multline*}
	
	The last term $\P_n^y(a_n(\tau) \vee d_n(\tau) \geq \varepsilon q)$ is upper bounded by $\P(\Pcal(\kappa \tau) \geq \varepsilon q)$ which we control using~\eqref{eq:deviation-P}. One can check that $h(x) \geq x$ for $x \geq e$, hence for $q$ such that $e\kappa \tau \leq \varepsilon q$ we obtain
	\[
		\P_n^y \left( a_n(\tau) \vee d_n(\tau) \geq \varepsilon q \right) \leq e^{-\varepsilon q}, \ q \geq e \kappa \tau / \varepsilon.
	\]
	
	In particular there exists a finite constant $\overline c$ such that $\P_n^y(a_n(\tau) \vee d_n(\tau) \geq \varepsilon q) \leq \overline c e^{-\varepsilon q}$ for all $q \geq 0$. On the other hand, $\{ \norm{x_n(\tau)} \geq q, a_n(\tau) \leq \varepsilon q \} \subset \{ \norm{x_n(0)} \geq (1-\varepsilon)q \}$ and according to~\eqref{eq:coupling-inequalities}, we also have the inclusion
	\[ \left\{ a_n(\tau) \vee d_n(\tau) \leq \varepsilon q \right\} \subset \left\{ x_{n,k}(\tau) \geq x'_{n,k}(\tau) - \varepsilon q, \norm{x_n(\tau)} \leq \norm{x_n(0)} + \varepsilon q \right\}, \]
	hence
	\begin{multline*}
		\P_n^y \left( \norm{x_n(\tau)} \geq q, x_{n,k}(\tau) \leq \varepsilon \pi_k \norm{x_n(\tau)}, a_n(\tau) \vee d_n(\tau) \leq \varepsilon q \right)\\
		\leq \P_n^y \left( \norm{x_n(0)} \geq (1-\varepsilon) q, x'_{n,k}(\tau) \leq \varepsilon \pi_k(\norm{x_n(0)} + \varepsilon q) + \varepsilon q \right)\\
		\leq \P_n^y \left( \norm{x_n(0)} \geq (1-\varepsilon) q, r'_{n,k}(\tau) \leq \varepsilon \pi_k (1 + \varepsilon/(1-\varepsilon)) + \varepsilon / (1-\varepsilon) \right).
	\end{multline*}

	Recalling that $\eta = 2\varepsilon / (\underline \pi(1-\varepsilon))$, one sees that the last term of the previous equation is upper bounded by $\indicator{\norm y \geq (1-\varepsilon) q} \P_n^y(r'_{n,k}(\tau) \leq \eta \pi_k)$ and we have further
	\[ \indicator{\norm y \geq (1-\varepsilon) q} \P_n^y \left( r'_{n,k}(\tau) \leq \eta \pi_k \right) \leq \indicator{\norm y \geq (1-\varepsilon) q} \P_n^y \left( \norm{r'_n(\tau) - \pi} \geq (1-\eta) \underline \pi \right). \]
	Since $0 < (1-\eta) \underline \pi < \varepsilon_0$,
Lemma~\ref{lemma:bound-closed} implies
	\[ \indicator{\norm y \geq (1-\varepsilon) q} \P_n^y \left( \norm{r'_n(\tau) - \pi} \geq (1-\eta) \underline \pi \right) \leq 2K \exp \left( -\frac{(1-\eta)^2 \underline \pi^2 (1-\varepsilon) q}{4K^2} \right). \]
	This proves the result, with for instance $c' = (2K)^{-2} \min(4K^2\varepsilon, (1-\eta)^2 \underline \pi^2 (1-\varepsilon))$ and $c = \overline c + 2K$.
\end{proof}

\begin{proof}
[Proof of Theorem~\ref{thm:ht-stationary-distribution}]

We first prove the convergence of the moments by induction on $r \geq 0$. The result is immediate for $r = 0$ so consider $r \geq 1$ and assume by induction hypothesis that $\E_n^{\nu_n}(\norm{X_n(0)}^s) \to s! / \alpha^s$ for every $0 \leq s \leq r-1$. Let $G_n$ be a geometrically distributed random variable with
parameter~$\rho_n$: Lemma~\ref{lemma:coupling-MM1} implies that
\[
\E_n^{\nu_n} \left( \norm{x_n(0)}^r \right) \geq \E \left( G_n^r \right),
\]
and using $n (1-\rho_n) \to \alpha$, it can be proved that
$\E((G_n/n)^r) \to r! / \alpha^r$. This provides a lower bound and so we only have to show that
$\limsup_{n \to +\infty} \E_n^{\nu_n}(\norm{X_n(0)}^r) \leq r!/\alpha^r$.
Let $m \geq 1$: summing the balance equations over the set
$\{y \in \N^K: \norm{y} \leq m - 1\}$ yields
	\begin{multline} \label{eq:balance}
		\lambda_n \P_n^{\nu_n}\left(\norm{x_n(0)} = m - 1\right) = \mu_n \P_n^{\nu_n}\left(\norm{x_n(0)} = m\right)\\
		- \sum_{k=1}^K \mu_{n, k} \P_n^{\nu_n} \left(\norm{x_n(0)} = m, x_{n, k}(0) = 0\right).
	\end{multline}
	
For any $M \geq 1$, writing $m^r = \sum_{s=0}^r \frac{r!}{s! (r-s)!} (m-1)^s$ gives
\[
\sum_{m = 1}^{M} m^r \P_n^{\nu_n}\left(\norm{x_n(0)} = m - 1\right) =
\sum_{s = 0}^{r} \frac{r!}{s! (r-s)!}
\E_n^{\nu_n} \left( \norm{x_n(0)}^s ; \norm{x_n(0)} \leq M-1 \right)
\]
	and so multiplying~\eqref{eq:balance} with $m^r / \mu_n$ on both sides and summing over $1 \leq m \leq M$, we obtain
\begin{multline*}
\sum_{s = 0}^{r} \frac{r!}{s! (r-s)!}
\rho_n \E_n^{\nu_n} \left( \norm{x_n(0)}^s ; \norm{x_n(0)} \leq M-1 \right)
= \E_n^{\nu_n} \left( \norm{x_n(0)}^r ; \norm{x_n(0)} \leq M \right)\\ -
\frac{1}{\mu_n} \sum_{k=1}^K \mu_{n, k}
\E_n^{\nu_n} \left( \norm{x_n(0)}^r ; \norm{x_n(0)} \leq M, x_{n,k}(0) = 0 \right).
\end{multline*}

Using $\E_n^{\nu_n} \left( \norm{x_n(0)}^r ; \norm{x_n(0)} \leq M \right) \geq \E_n^{\nu_n} \left( \norm{x_n(0)}^r ; \norm{x_n(0)} \leq M - 1 \right)$, other simple inequalities and isolating the terms corresponding to $s = r$ and $s=r-1$ in the previous sum, we end up with
\begin{multline*}
(1-\rho_n)
\E_n^{\nu_n} \left( \norm{x_n(0)}^r ; \norm{x_n(0)} \leq M - 1 \right)
\leq r \E_n^{\nu_n} \left( \norm{x_n(0)}^{r-1} \right) + r! \sum_{s = 0}^{r - 2} \E_n^{\nu_n} \left( \norm{x_n(0)}^s \right)\\
+ \sum_{k=1}^K \E_n^{\nu_n} \left( \norm{x_n(0)}^r ; x_{n,k}(0) = 0 \right).
\end{multline*}

Letting $M \to +\infty$ and dividing by $n^r (1-\rho_n)$ gives
\begin{multline*}
\E_n^{\nu_n} \left( \norm{X_n(0)}^r \right) \leq \frac{r}{n (1-\rho_n)} \E_n^{\nu_n} \left( \norm{X_n(0)}^{r - 1} \right) +
\frac{r!}{n (1-\rho_n)} \sum_{s = 0}^{r - 2} n^{s-(r-1)} \E_n^{\nu_n} \left( \norm{X_n(0)}^s \right)\\
+ \frac{1}{n^r (1-\rho_n)} \sum_{k=1}^K
\E_n^{\nu_n} \left( \norm{x_n(0)}^r ; x_{n,k}(0) = 0 \right).
\end{multline*}

Since $n (1-\rho_n) \to \alpha > 0$, by induction hypothesis the first term of the above upper bound converges to $r! / \alpha^r$ while the second term vanishes, thanks to the terms $n^{s-(r-1)}$ that go to $0$ for $s \leq r-2$. We now show that the last term also vanishes, thus concluding the proof. To this end we prove that $\E_n^{\nu_n}\left(\norm{x_n(0)}^r ; x_{n, k}(0) = 0\right) \to 0$
for all $1 \leq k \leq K$. Let $k \in \{1, \ldots, K\}$ and $M \geq 0$: we distinguish the two events $\{\norm{x_n(0)} \leq M\}$ and $\{\norm{x_n(0)} \geq M+1\}$. On the one hand we have
\[ \E_n^{\nu_n}\left(\norm{x_n(0)}^r ; \norm{x_n(0)} \leq M, x_{n, k}(0) = 0\right) \leq
M^r \P_n^{\nu_n}\left( \norm{x_n(0)} \leq M \right) \leq M^r \left(1-(\rho_n)^M \right) \]
invoking Lemma~\ref{lemma:coupling-MM1} to get the last bound. On the other hand, Lemma~\ref{lemma:constant} gives
\begin{align*}
\E_n^{\nu_n}\left(\norm{x_n(0)}^r ; \norm{x_n(0)} \geq M + 1, x_{n, k}(0)
= 0\right) & = \sum_{m \geq M+1} \P_n^{\nu_n} \left( \norm{x_n(0)}^r \geq m, x_{n, k}(0) = 0 \right) \\
&\leq \sum_{m \geq M} ce^{-c' m^{1/r}}.
\end{align*}

Since $\rho_n \to 1$, letting $n \to +\infty$ for a given fixed $M$ gives
\[ \limsup_{n \to +\infty} \E_n^{\nu_n}\left(\norm{x_n(0)}^r ; x_{n, k}(0) = 0\right) \leq \sum_{m \geq M} ce^{-c' m^{1/r}}. \]

Letting $M \to +\infty$ achieves to prove that integer moments converge.
\\

We now prove the weak convergence result. Since $\E_n^{\nu_n}(\norm{X_n(0)}^2) \to 2/\alpha^2$, the sequence
$(\norm{X_n(0)}, n \geq 1)$ is uniformly integrable and tight. Let $X$ be any accumulation point and assume without loss of generality
that $X_n(0) \Rightarrow X$: we will prove that $X = E \pi$, thus proving the desired result.
	
Since $n (1-\rho_n) \to \alpha$, $\norm{X}$ is stochastically lower
bounded by an exponential random variable with parameter~$\alpha$
thanks to Lemma~\ref{lemma:coupling-MM1}.
On the other hand, since $X_n(0) \Rightarrow X$,
$\E_n^{\nu_n}(\norm{X_n(0)}) \to 1/\alpha$ and $(\norm{X_n(0)}, n \geq 1)$
is uniformly integrable, we get convergence of the means and so
$\E(\norm{X}) = 1/\alpha$.
In summary, if $E$ is an exponential random variable with
parameter~$\alpha$, then $\norm{X}$ is stochastically lower bounded by~$E$
and $\E(\norm{X}) = \E(E)$: hence $\norm{X}$ and $E$ must be equal
in distribution.
	
	In particular, $\P(\norm{X} = 0) = 0$ and so the continuous-mapping theorem implies that $\norm{R_n(0) - \pi} \Rightarrow \norm{X / \norm{X} - \pi}$. Let $0 < \varepsilon < 2 \varepsilon_0$ and $\tau = \tau(\varepsilon / (4K))$: since
	\[ \P_n^{\nu_n} \left( \norm{R_n(0) - \pi} \geq \varepsilon \right) = \P_n^{\nu_n} \left( \norm{r_n(\tau) - \pi} \geq \varepsilon \right) = 0 \]
	Lemma~\ref{lemma:bound-open} implies that $\P_n^{\nu_n} \left( \norm{R_n(0) - \pi} \geq \varepsilon \right) \to 0$ and so $\P(\norm{X / \norm{X} - \pi} \geq \varepsilon) = 0$. Since $\varepsilon < 2 \varepsilon_0$ is arbitrary, letting $\varepsilon \to 0$ shows that $X = \norm{X} \pi$ and since $\norm{X}$ is an exponential random variable with parameter~$\alpha$ this proves the result.
\end{proof}

\section{Convergence of the sojourn times}
\label{sec:sojourn-time}

Throughout this section the heavy-traffic assumption of Section~\ref{sub:main-results} continues to be
in force, and we moreover assume that $\alpha > 0$ so that $\rho_n < 1$ for $n$ large enough.
We fix a sequence $(y_n)$ in $\N^K$ with $\varrho(y_n) \to 0$
and $\norm{y_n} / n \to b \in (0,\infty)$, and we return to the network
description of~$x_n$, see Section~\ref{sub:networks}.
In the sequel we implicitly consider~$x_n$ under $\P_n^{y_n}$ and write
$\Rightarrow$ for weak convergence under $\P_n^{y_n}$.

For $n \geq 1$, we pick an initial user of the $n$th system uniformly
at random, which we refer to as the \emph{tagged user}, and denote
by~$E_n$ its initial service requirement (which is exponentially
distributed with unit mean), by~$\xi_n$ its trajectory and by~$\chi_n$
its sojourn time, so the following relations hold:
\[ \chi_n = \inf \left\{ t \geq 0: s_n(t) = E_n \right\} \ \text{ with } \ s_n(t) = \int_0^{t} \frac{\mu_{n, \xi_n(u)}}{1 \vee x_{n,\xi_n(u)}(u)} du. \]

Note that $s_n(t)$ for $t \leq \chi_n$ is the service received by the
tagged user up to time~$t$.  We want to show that
$n^{-1} \chi_n \Rightarrow b E / \lambda$ where in this section $E$ is
an exponential random variable with unit mean.
Introducing the rescaled service process $S_n(t) = s_n(nt)$,
we have $n^{-1} \chi_n = \inf \left\{ t \geq 0: S_n(t) = E_n \right\}$.
Let $s'_n$ be defined as follows:
\[ s_n'(t) = \frac{1}{\norm{y_n}} \int_0^{t} \frac{\mu_{n, \xi_n(u)}}{\pi_{\xi_n(u)}} du \]
and define $S'_n(t) = s'_n(nt)$. Let finally $S = (b^{-1} \lambda t, t \geq 0)$.

\begin{lemma}
\label{lemma:s'}
	The sequence of processes $(S'_n)$ under $\P_n^{y_n}$ converges weakly to~$S$.
\end{lemma}

\begin{proof}
	Since $S'_n$ and $S$ are strictly increasing and continuous, we only have to show that finite-dimensional distributions converge, see for instance Jacod and Shiryaev~\cite[Theorem~VI.$3.37$]{Jacod03:0}. Since the law of $\xi_n$ does not depend on~$n$, we can couple all the processes $(S'_n, n \geq 1)$ on a common probability space by taking $\xi_n = \xi_1$. Then we have
	\[ S'_n(t) = \frac{1}{\norm{y_n}} \int_0^{nt} \frac{\mu_{n, \xi_1(u)}}{\pi_{\xi_1(u)}} du = \frac{1}{\norm{y_n}} \sum_{k=1}^K \frac{\mu_{n,k}}{\pi_{k}} \int_0^{nt} \indicator{\xi_1(u) = k} \, du \]
	and so
	\[ \left|S'_n(t) - \frac{\lambda t}{b}\right| \leq \frac{n}{\norm{y_n}} \frac{1}{n} \left| \sum_{k=1}^K \frac{\mu_{n,k}}{\pi_{k}} \int_0^{nt} \indicator{\xi_1(u) = k} \, du - n \mu_n t \right| + \left|\frac{n \mu_n}{\norm{y_n}} - \frac{\lambda}{b}\right| t. \]
	The second term goes to~0 by assumption while the first term is easily seen to vanish using the ergodic theorem. This shows that $S_n'(t) \to S(t)$ almost surely, for every $t \geq 0$, which readily implies convergence of the finite-dimensional distributions. This proves the result.
\end{proof}

\begin{lemma}
The sequence of processes $(S_n)$ under $\P_n^{y_n}$ converges weakly to~$S$.
\end{lemma}

\begin{proof}
	Since $S'_n \Rightarrow S$ by Lemma~\ref{lemma:s'}, Theorem~3.1
in Billingsley~\cite{Billingsley99:0} shows that it suffices to prove
that for any $\varepsilon > 0$ and any $t_0 \geq 0$,
	\[ \lim_{n \to +\infty} \P_n^{y_n} \left( \sup_{0 \leq t \leq nt_0} \left|s_n(t) - s'_n(t) \right| \geq \varepsilon \right) = 0. \]
	
	By definition we have
	\begin{align*}
		\left|s_n(t) - s'_n(t) \right| & = \left| \int_0^{t} \frac{\mu_{n, \xi_n(u)}}{x_{n,\xi_n(u)} \vee 1} du - \frac{1}{\norm{y_n}} \int_0^{t} \frac{\mu_{n, \xi_n(u)}}{\pi_{\xi_n(u)}} du \right|\\
		& \leq \frac{1}{\norm{y_n}} \int_0^{t} \frac{\mu_{n, \xi_n(u)}}{\pi_{\xi_n(u)} (x_{n,\xi_n(u)} \vee 1)} \left|x_{n,\xi_n(u)} \vee 1 - \pi_{\xi_n(u)} \norm{y_n} \right| du.
	\end{align*}
	
	Fix in the sequel $c = 2 + \sup_{n \geq 1} (\norm{y_n}/n)$ and $\delta > 0$ such that
	\[ \delta \leq \underline \pi, \ \ \delta < b \underline \pi / c, \ \ \eta < 2 \varepsilon_0 \ \text{ and } \ \sup_{n \geq 1} 
\left( \frac{\mu_n \delta c n^2 t_0}{\norm{y_n} \underline \pi (\underline \pi \norm{y_n} - \delta c n)} \right) \leq \varepsilon / 2, \]
	where $\eta = \underline \pi \delta^2 / 32$. Let $F_n$ be the event
	\[ F_n = \left\{ T^\uparrow(r_n - \pi, \delta) \geq nt_0 \ \right\} \cap \left\{ T^\uparrow(\norm{x_n} - \norm{y_n}, n \delta) \geq nt_0 \ \right\} \]
	where $\norm{x_n} - \norm{y_n}$ refers to the process $(\norm{x_n(t)} - \norm{y_n}, t \geq 0)$. Since $\delta \leq \underline \pi$ we have $T^\uparrow(r_n - \pi, \delta) \leq \widetilde T_0(x_n)$ by~\eqref{eq:lower-bound-tilde} and so $x_{n,k}(t) > 0$ for all $t \leq n t_0$ and $1 \leq k \leq K$ in~$F_n$, in particular $x_{n,k}(t) \vee 1 = x_{n,k}(t)$. Thus, we also have in~$F_n$
	\begin{align*}
		\left| x_{n,k}(t) \vee 1 - \pi_k \norm{y_n} \right| = \left| x_{n,k}(t) - \pi_k \norm{y_n} \right| & \leq \left| x_{n,k}(t) - \pi_k \norm{x_n(t)} \right| + \pi_k \left| \norm{x_n(t)} - \norm{y_n} \right|\\
		& \leq \left( \norm{y_n} + \delta n \right) \norm{r_n(t) - \pi} + \delta n \leq \delta c n
	\end{align*}
	for all $t \leq n t_0$ and $1 \leq k \leq K$. In particular, $x_{n,k}(t) \geq \underline \pi \norm{y_n} - \delta c n$ which is positive for $n$ large enough since $\delta c < b \underline \pi$. Plugging in these different bounds, we obtain the following inequality, valid in~$F_n$ for any $t \leq n t_0$:
	\begin{align*}
		\left|s_n(t) - s'_n(t) \right| & \leq \frac{1}{\norm{y_n}} \int_0^{t} \frac{\mu_n}{\underline \pi (\underline \pi \norm{y_n} - \delta c n)} \delta c n du \leq \frac{\mu_n \delta c n^2 t_0}{\norm{y_n} \underline \pi (\underline \pi \norm{y_n} - \delta c n)} \leq \varepsilon / 2.
	\end{align*}
	
	Thus we obtain $|s_n(t) - s'_n(t)| < \varepsilon$ in $F_n$ and therefore
	\begin{multline*}
		\P_n^{y_n} \left( \sup_{0 \leq t \leq nt_0} \left|s_n(t) - s'_n(t) \right| \geq \varepsilon \right) = \P_n^{y_n} (F_n^c) \leq \P_n^{y_n} \left( T^\uparrow(\norm{x_n} - \norm{y_n}, n \delta) \leq nt_0 \ \right)\\
		+ \P_n^{y_n} \left( T^\uparrow(r_n - \pi, \delta) \leq nt_0 \wedge T^\uparrow(\norm{x_n} - \norm{y_n}, n \delta) \ \right).
	\end{multline*}
	Since $T^\uparrow(\norm{x_n} - \norm{y_n}, n \delta) \leq T^\downarrow(x_n, \norm{y_n} - n \delta)$, the second term can be shown to go to~0 by Proposition~\ref{prop:control-homogenization} (using that $\varrho(y_n) \leq \eta$ for $n$ large enough since $\varrho(y_n) \to 0$). The first can also be shown to go to~0, since
	\[ \P_n^{y_n} \left( T^\uparrow(\norm{x_n} - \norm{y_n}, n \delta) \leq nt_0 \ \right) = \P_n^{y_n} \left( T^\uparrow(\norm{X_n} - \norm{X_n(0)}, \delta) \leq t_0/n \ \right) \]
	and the convergence $X_n \Rightarrow \underline B \pi$ of Theorem~\ref{thm:cv-process} implies that $T^\uparrow(\norm{X_n} - \norm{X_n(0)}, \delta) \Rightarrow T^\uparrow(B - b, \delta)$ which is strictly positive. The result is proved.
\end{proof}

\begin{corollary}
\label{cor:sojourn}
	The sequence of random variables $(n^{-1} \chi_n)$ under $\P_n^{y_n}$ converges weakly to $b E / \lambda$.
\end{corollary}

\begin{proof}
	Since $S_n \Rightarrow S$ and $S$ is strictly increasing and continuous with $S(0) = 0$, the continuous-mapping theorem implies that $S_n^{-1} \Rightarrow S^{-1}$, see for instance Whitt~\cite{Whitt02:0}. Moreover, the $E_n$'s are identically distributed and so $E_n \Rightarrow E$. Since $S^{-1}$ is deterministic we get the joint convergence $(S_n^{-1}, E_n) \Rightarrow (S^{-1}, E)$ and the continuous-mapping theorem then implies that $S_n^{-1}(E_n) \Rightarrow S^{-1}(E)$, since $S^{-1}$ is a continuous function. Since $S_n^{-1}(E_n) = n^{-1} \chi_n$ and $S^{-1}(E) = b E / \lambda$ this proves the result.
\end{proof}

\begin{corollary}
	The sequence of random variables $(n^{-1} \chi_n)$ under $\P_n^{\nu_n}$ converges weakly to $E' E / \lambda$ with $E'$ an exponential random variable independent of~$E$ and with parameter~$\alpha$.
\end{corollary}

\begin{proof}
	This is a consequence of Corollary~\ref{cor:sojourn} together with the fact that $n^{-1} x_n(0)$ under $\P_n^{\nu_n}$ converges weakly to $E' \pi$ by Theorem~\ref{thm:ht-stationary-distribution}.
\end{proof}

The above corollary is similar to heavy-traffic results for the
sojourn time distribution ordinary Processor-Sharing queue,
see for instance Sengupta~\cite{Sengupta92:0}, Yashkov~\cite{Yashkov93:0},
and Zwart \& Boxma~\cite{Zwart00:0}.
These results may be intuitively explained by the snapshot principle, see Reiman~\cite{Reiman82:0}: in
heavy-traffic conditions the total number of users in the system hardly
varies over the time scale of a sojourn time.
Thus each individual user sees a service rate that is random,
determined by the inverse of the total number of users in stationarity
which has an asymptotically exponential distribution in heavy traffic,
but nearly constant over the duration of its sojourn time.

It is worth emphasizing that although in the present model the users
within each of the individual nodes are served in a Processor-Sharing
manner, at any given time the service rates of users may strongly vary
across nodes.
Due to the homogenization property, however, the empirical distribution
of the location of each individual user over the course of a long
sojourn time in a heavy-traffic regime will be close to the stationary
distribution~$\pi$.
Hence each individual user will see a $\pi$-weighted average of the
service rates in the various nodes, which is only affected by the
exponentially distributed total number of users in the entire system,
just like in an ordinary Processor-Sharing queue.

\bibliographystyle{plain}

\end{document}